\documentclass[10pt,twoside,leqno]{amsart}
\usepackage{amsfonts}
\usepackage{amsmath}
\usepackage{mathtools}
\usepackage{amsthm}
\usepackage{amssymb}
\usepackage{pifont}
\newcommand{\cmark}{\ding{51}}
\newcommand{\xmark}{\ding{55}}
\usepackage{mathrsfs}
\usepackage[numbers]{natbib}
\usepackage[fit]{truncate}
\usepackage{geometry}
\usepackage{hyperref}
\usepackage{bm}
\usepackage{tikz-cd}
\usepackage[makeroom]{cancel}
\usepackage{xcolor,colortbl}
\usepackage{hhline}
\usepackage{float}
\usepackage{enumitem} 
\usepackage{stackengine}
\newcommand\xrowht[2][0]{\addstackgap[.5\dimexpr#2\relax]{\vphantom{#1}}}
\usepackage{array, multirow}
\usepackage{etoolbox}
\makeatletter
\patchcmd{\@sect}{\@svsec\@empty}{\@svsec\@empty\stepcounter{#1}}{}{}
\makeatother

\newcommand{\R}{\mathbb{R}}
\newcommand{\C}{\mathbb{C}}

\newcommand{\cm}[1]{\ignorespaces}

\theoremstyle{plain}
\newtheorem{theorem}{Theorem}[section]
\newtheorem{corollary}[theorem]{Corollary}
\newtheorem{lemma}[theorem]{Lemma}

\newtheorem{proposition}[theorem]{Proposition}

\theoremstyle{definition}
\newtheorem{definition}[theorem]{Definition}
\newtheorem{example}[theorem]{Example}

\newtheorem{remark}[theorem]{Remark}
\numberwithin{equation}{section}

\begin{document}

\title[Balanced  Hermitian structures on almost abelian Lie algebras]{Balanced  Hermitian structures \\ on almost abelian Lie algebras}
\author{Anna Fino}
\address[A. Fino]{Dipartimento di Matematica ``G. Peano''\\
	Universit\`a di Torino\\
	Via Carlo Alberto 10\\
	10123 Torino, Italy} \email{annamaria.fino@unito.it}

\author{Fabio Paradiso}
\address[F. Paradiso]{Dipartimento di Matematica ``G. Peano''\\
	Universit\`a di Torino\\
	Via Carlo Alberto 10\\
	10123 Torino, Italy} \email{fabio.paradiso@unito.it}


\begin{abstract}
We study  balanced  Hermitian structures on almost abelian Lie algebras, i.e.\ on Lie algebras with a codimension-one abelian ideal. In particular, we classify six-dimensional almost abelian Lie algebras  which carry a  balanced structure. It has been conjectured in \cite{FV} that a compact complex manifold admitting both a balanced metric and an SKT metric necessarily has a K\"ahler metric: we prove this conjecture for compact almost abelian solvmanifolds with left-invariant complex structures. Moreover, we investigate the behaviour of the flow of balanced metrics introduced in \cite{BV} and of the anomaly flow \cite{PPZ2} on almost abelian Lie groups. In particular, we show that the anomaly flow preserves the balanced condition and that locally conformally K\"ahler metrics are fixed points.
\end{abstract}

\maketitle

\section{Introduction}

Let $(M,J)$ be a complex manifold of real dimension $2n$, $n \geq 3$. A Hermitian metric $g$ is called \emph{balanced} (or \emph{semi-K\"ahler}) if its associated fundamental form $\omega=g(J\cdot,\cdot)$ is coclosed, namely $d^*\omega=0$, or equivalently $d \omega^{n -1} =0$.  Balanced metrics were introduced in \cite{Mic} and provide a generalization of K\"ahler metrics. Another possible generalization is given  by \emph{strong K\"ahler with torsion} (SKT, also known as \emph{pluriclosed}) metrics, satisfying $\partial \overline\partial \omega =0$ (see for instance \cite{FT1, FPS}). SKT metrics have natural applications in type II string theory and $2$-dimensional supersymmetric $\sigma$-models \cite{GHR, Str} and are closely related to generalized K\"ahler structures \cite{Gua,  Gua2, AG, FT, FP}. The balanced and SKT conditions are transversal to each other, in the sense that a Hermitian metric which is balanced and SKT is necessarily K\"ahler (see \cite{AI}). Moreover, it has been conjectured that a compact complex manifold admitting both balanced and SKT metrics carries a K\"ahler metric as well \cite{FV}. This has been confirmed on some special classes of compact complex manifolds: twistor spaces of compact anti-self-dual Riemannian manifolds \cite{Ver}, non-K\"ahler manifolds belonging to the Fujiki class $\mathcal{C}$ \cite{Chi}, a family of non-K\"ahler Calabi-Yau threefolds   constructed in \cite{FLY}, generalized Calabi-Gray manifolds \cite{Fei}, Oeljeklaus-Toma manifolds \cite{Oti}, $2$-step nilmanifolds \cite{FV1} and six-dimensional solvmanifolds with holomorphically trivial canonical bundle \cite{FV} when the complex structure is left-invariant.

By \cite{AG}, every Hermitian metric on a unimodular \emph{complex} Lie algebra $\mathfrak{g}$ (which can be identified with a real Lie algebra endowed with an $\text{ad}$-invariant complex structure) is balanced. If, moreover, $\mathfrak{g}$ is semisimple, then $\omega^{n-1}$ is also $\partial \overline \partial$-exact (see \cite{Yac}).

In \cite{Gau}, it was proven that a Hermitian manifold $(M,J,g)$ admits a $1$-parameter family $\{\nabla^{\tau}\}_{\tau \in \R}$ of canonical Hermitian connections, satisfying $\nabla^{\tau}J=0$, $\nabla^{\tau}g=0$. Explicitly, $\nabla^{\tau}$ is determined by
\begin{equation} \label{nablatau}
g\left(\nabla_{X}^{\tau} Y, Z\right)=g\left(\nabla_{X}^g Y, Z\right)+\frac{1-\tau}{4} T(X, Y, Z)+\frac{1+\tau}{4} C(X, Y, Z), \quad X,Y,Z \in \Gamma(TM),
\end{equation}
where $\nabla^g$ is  the Levi-Civita connection associated with $g$ and  $T$ and $C$ are given respectively by
\[
T (X, Y,Z) =-Jd\omega (X, Y, Z) =d\omega(J XJ Y,JZ), \;\; C (X, Y, Z)=-d\omega(J X, Y,Z),  \;\; X,Y,Z \in \Gamma(TM) .
\]
Notable connections in this family include the \emph{Bismut connection}  $\nabla^B$ \cite{Bismut} and the \emph{Chern connection} $\nabla^C$, corresponding to $\tau=-1$ and $\tau=1$ respectively.
In particular, $\nabla^B$  is the unique Hermitian connection with totally skew-symmetric torsion, given by  the 3-form $T$.

For each canonical Hermitian connection $\nabla^{\tau}$, one can define its associated \emph{Ricci form} $\rho^\tau \in \Gamma(\Lambda^2T^*M)$, locally given by
\[
\rho^{\tau}(X,Y)=-\frac{1}{2} \sum_{i=1}^{2n} g\left(R^\tau(X,Y)e_i,Je_i\right),\quad X,Y \in \Gamma_{\text{loc}}(TM),
\]
where $\{e_1,\ldots,e_{2n}\}$ is a local $g$-orthonormal frame and $R^{\tau}(X,Y)=[\nabla^{\tau}_X,\nabla^{\tau}_Y]-\nabla^{\tau}_{[X,Y]}$ denotes the curvature of $\nabla^{\tau}$. By \cite{AI}, the Chern-Ricci form $\rho^C$  and the Bismut-Ricci form $\rho^B$ are related by 
\begin{equation} \label{ChernBismut}
\rho^C=\rho^B-dJ\theta,
\end{equation}
where $\theta$ is the \emph{Lee form}, namely the unique $1$-form $\theta$ such that
$d\omega^{n-1}= \theta \wedge \omega^{n-1}$. Equivalently, $\theta=Jd^*\omega=-J*d*\omega$, where $*$ denotes the Hodge star operator of $g$ (for the orientation induced by $J$) and $J$ acts on $1$-forms by $J\alpha=-\alpha(J\cdot)$. The vanishing of $\theta$ is equivalent to the balanced condition and implies $\rho^B=\rho^C$.

Since $\nabla^B$ is Hermitian, its restricted holonomy group, denoted by $\text{Hol}^0(\nabla^B)$, is contained in the unitary group $\text{U}(n)$. Hermitian  structures satisfying $\text{Hol}^0(\nabla^B) \subseteq \text{SU}(n)$, or equivalently $\rho^B=0$, are known in literature as \emph{Calabi-Yau with torsion} and appear in heterotic string theory, related to the Hull-Strominger system in six dimensions (\cite{Hull, Str, LY}).

A (solvable) Lie algebra $\mathfrak{g}$ is called \emph{almost abelian} if it contains an abelian ideal $\mathfrak{n}$ of codimension one. If $\mathfrak{g}$ is non-nilpotent, this ideal is unique and coincides with the nilradical of $\mathfrak{g}$.

Hermitian structures $(J, g)$  on $\mathfrak{g}$ can be characterized in terms of the behaviour of the matrix associated with $\text{ad}_{e_{2n}}\rvert_{\mathfrak{n}}$, with respect to some \emph{adapted} unitary basis $\{e_1,\ldots,e_{2n}\}$  of $(\mathfrak{g}, J, g)$  such that $\mathfrak{n}=\text{span}\left<e_1,\ldots,e_{2n-1}\right>$, $Je_i=e_{2n+1-i}$, $i=1,\ldots,n$, $\mathfrak{k} \coloneqq \mathfrak{n}^{\perp_g}=\R e_{2n}$. See \cite{LV, LW, FP,AO,AL} for the characterization of K\"ahler, locally conformally K\"ahler (LCK) and SKT metrics on almost abelian Lie algebras.
In this paper we prove  that balanced  Hermitian structures are characterized by the conditions $[Z,J Z] \in \R JZ$,
 $\text{tr} \left(\text{ad}_{Z} \rvert_{\mathfrak{n} \cap J\mathfrak{n}}\right)=0$,
  for any $Z \in \mathfrak{k}$. 
In particular, we show that a unimodular balanced Hermitian almost abelian Lie algebra is always decomposable. Moreover,  we prove that  a compact  almost abelian solvmanifold with a left-invariant complex structure admitting both a balanced metric and an SKT metric necessarily has a K\"ahler metric, where by compact  almost abelian solvmanifold we mean the compact quotient of a simply connected almost abelian Lie group by a lattice.

In \cite{FPS}, it was proven that a Hermitian structure on a $2n$-dimensional nilpotent Lie algebra is balanced if and only if 
$\text{Hol}^0(\nabla^B) \subseteq \text{SU}(n)$. The same result holds more in general   for  compact  complex manifolds   of complex dimension $n$ admitting   a closed $(n,0)$-form (see \cite[Proposition 3.6]{Gar}).
In this paper, we prove that the analogous equivalence holds for unimodular almost abelian Lie algebras, also when the complex structure does not admit closed $(n,0)$-forms.
We also recall that, by \cite[Theorem 15]{AV}, such a reduction of the restricted Bismut holonomy group holds for balanced metrics on unimodular Lie algebras endowed with \emph{abelian} complex structures $J$, namely such that $[JX,JY]=[X,Y]$ for all vector fields $X,Y$.

A parabolic flow involving balanced metrics  has been introduced  in \cite{BV}.   This flow of Hermitian metrics, which we shall call \emph{balanced flow}, preserves  the balanced condition and reduces to the well-known \emph{Calabi flow} (see \cite{Cal, Cal2}) for K\"ahler initial data. In \cite{BV1}, the same authors proved that this flow is stable around Ricci-flat K\"ahler metrics. In this paper, we study left-invariant solutions on Lie groups, introducing the notion of \emph{semi-algebraic soliton}, which is related to self-similar solutions. We then rephrase the problem in terms of \emph{bracket flows} in the almost abelian case, proving results about the long-time behaviour of the flow.

The \emph{anomaly flow}, introduced in \cite{PPZ, PPZ2} is a flow of $(2, 2)$-forms on a  complex  $3$-fold which arises from the Hull-Strominger system \cite{Hull,Str}, in the context of heterotic string theory. In its simplest formulation, the anomaly flow is a flow for a Hermitian metric on a  complex $3$-fold with holomorphically trivial canonical bundle, depending on the choice of a Hermitian connection on the underlying manifold. In the case of the Chern connection, it was shown in \cite{PPZ} that the flow preserves the conformally balanced condition. In \cite{PPZ2}, the flow was studied for unimodular complex Lie groups, while recently it was considered on $2$-step nilpotent Lie groups \cite{PPZ3},  showing that   it  is well-defined for any canonical Hermitian connection $\nabla^\tau$ and preserves the balanced and LCK conditions. In this paper, we prove that, in the context of left-invariant structures on almost abelian Lie groups and for any choice of canonical Hermitian connection $\nabla^\tau$,  the flow preserves the balanced condition and LCK metrics are fixed points, namely they provide constant solutions of the flow.

The paper is structured as follows: in Section \ref{Herm} we introduce the framework of Hermitian almost abelian Lie algebras.  In particular, in terms of algebraic data associated to a fixed adapted unitary basis, we provide a formula  for the Lee form $\theta$ and we investigate Hermitian almost abelian Lie algebras admitting closed $(n,0)$-forms, where $n$ denotes the complex dimension of the Lie algebra.

In Section \ref{balancedmetrics}, we characterize balanced Hermitian almost abelian Lie algebras and  prove the equivalence of the balanced and Calabi-Yau with torsion conditions in the unimodular case. As a consequence  we show that  the aforementioned conjecture regarding balanced and SKT metrics holds for compact almost abelian solvmanifolds of any dimension endowed with left-invariant complex structures.

Section \ref{6-dim} is devoted to the classification of six-dimensional non-nilpotent almost abelian Lie algebras admitting balanced Hermitian structures. In particular, up to isomorphism, we obtain four classes of six-dimensional non-nilpotent unimodular almost abelian Lie algebras, all of which are decomposable as the direct sum of a $5$-dimensional almost abelian Lie algebra and $\R$. We also investigate whether the corresponding simply connected Lie groups admit compact quotients.

Finally, in Sections \ref{sec_bal} and \ref{sec_anom}  we study the balanced flow and the anomaly flow on almost abelian Lie groups, respectively, presenting the aforementioned results. 

\smallskip

\emph{Acknowledgements}. The authors would like to thank Lucio Bedulli, Mattia Pujia and Luigi Vezzoni for useful comments and discussions. They would also like to thank an anonymous referee for useful remarks. The paper is supported by Project PRIN 2017 ``Real and complex manifolds: Topology, Geometry and Holomorphic Dynamics'' and by GNSAGA of INdAM.

\section{Hermitian almost abelian Lie algebras} \label{Herm}
Assume $(J,g)$ is a Hermitian structure on a $2n$-dimensional almost abelian Lie algebra $\mathfrak{g}$. In particular, $J$ is an (integrable) complex structure on $\mathfrak{g}$. Fix an abelian ideal of codimension one $\mathfrak{n}$ and denote by $\mathfrak{n}_1 \coloneqq \mathfrak{n} \cap J\mathfrak{n}$ the maximal $J$-invariant subspace of $\mathfrak{n}$. Then there exists a unitary basis $\{e_1,\ldots,e_{2n}\}$ of $\mathfrak{g}$, such that $\mathfrak{n}=\text{span}\left<e_1,\ldots,e_{2n-1}\right>$, $\mathfrak{n}_1=\text{span}\left<e_2,\ldots,e_{2n-1}\right>$, $Je_i=e_{2n+1-i}$, $i=1,\ldots,n$. As shown in \cite{LV}, the matrix $B$ associated with $\text{ad}_{e_{2n}}\rvert_{\mathfrak{n}}$ in this basis is of the form
\begin{equation} \label{B}
	B=\begin{pmatrix} a & 0 \\ v & A \end{pmatrix}, \quad a \in \R,\,v \in \mathfrak{n}_1,\, A \in \mathfrak{gl}(\mathfrak{n}_1,J_1),
\end{equation}
where $J_1 \coloneqq J\rvert_{\mathfrak{n}_1}$ and $\mathfrak{gl}(\mathfrak{n}_1,J_1)$ denotes endomorphisms of $\mathfrak{n}_1$ commuting with $J_1$.
We denote $\mathfrak{k}\coloneqq \mathfrak{n}^{\perp_g}=\R e_{2n}$ and we say that the unitary basis $\{e_1,\ldots,e_{2n}\}$ is \emph{adapted} to the splitting $\mathfrak{g} = J\mathfrak{k} \oplus \mathfrak{n}_1 \oplus \mathfrak{k}$. The triple $(a,v,A)$ characterizes the Hermitian structure $(J,g)$ completely so that, in the following, we often denote a Hermitian almost abelian Lie algebra by $(\mathfrak{g}(a,v,A),J,g)$ to keep track of the algebraic data $(a,v,A)$.

For Hermitian structures on generic Lie algebras, we have
\begin{lemma} \label{leeform_gen}
		Let $(\mathfrak{g},J,g)$ be a Hermitian Lie algebra endowed with an orthonormal basis $\{e_1,\ldots,e_{2n}\}$. Then, its associated Lee form $\theta$ is given by
		\[
		\theta(X)=-\operatorname{tr} \text{\normalfont ad}_X + \frac{1}{2} g\left(\sum_{k=1}^{2n} [e_k,Je_k],JX\right),\quad X \in \mathfrak{g}.
		\]
\end{lemma}
\begin{proof}
		To compute $\theta$ we use the formula $d^*=-\sum_{k=1}^{2n} \iota_{e_k} \nabla^g_{e_k}$ for the codifferential, where $\{e_1,\ldots,e_{2n}\}$ denotes some orthonormal basis and $\nabla^g$ denotes the Levi-Civita connection for $g$, which satisfies the Koszul formula
		\begin{equation} \label{koszul}
			2g(\nabla^g_X Y, Z)=g([X,Y],Z)-g([Y,Z],X)-g([X,Z],Y),\quad X,Y,Z \in \mathfrak{g}.
		\end{equation}
		Then, for all $X \in \mathfrak{g}$, one can compute
		\begin{align*}
			\theta(X)=&-\sum_{k=1}^{2n} \left( g(\nabla^g_{e_k} e_k,X) + g(\nabla^g_{e_k} JX, Je_k) \right) \\
			=& - \frac{1}{2} \sum_{k=1}^{2n} \left( 2g([X,e_k],e_k) + g(J[JX,e_k],e_k) - g([JX,Je_k],e_k) - g([e_k,Je_k],JX) \right) \\
			=& -\operatorname{tr} \text{ad}_X + \frac{1}{2} \operatorname{tr} [\text{ad}_{JX},J] + \frac{1}{2} \sum_{k=1}^{2n} g([e_k,Je_k],JX).
		\end{align*}
		The claim then follows by noticing that the commutator of two endomorphisms is traceless.
\end{proof}

\begin{corollary} \label{leeform}
Let $(\mathfrak{g},J,g)$ be a Hermitian almost abelian Lie algebra endowed with an adapted unitary basis $\{e_1,\ldots,e_{2n}\}$, determining the algebraic data $(a,v,A)$. Then, its associated Lee form is
\begin{equation} \label{Lee_avA}
	\theta=(Jv)^{\flat} - (\operatorname{tr} A)e^{2n},
\end{equation}
where $(\cdot)^\flat \colon \mathfrak{g} \to \mathfrak{g}^*$ denotes the musical isomorphism induced by $g$ and $e^{2n}$ is the dual $1$-form associated with $e_{2n}$.
\end{corollary}
\begin{proof}
In terms of $(a,v,A)$, we can use the previous lemma and compute
\begin{align*}
		\theta(e_1)&=0, \\
		\theta(X)&=g(-ae_1-v,JX)=g(ae_{2n}+Jv,X)=g(Jv,X),\quad X \in \mathfrak{n}_1, \\
		\theta(e_{2n})&=-(a+\operatorname{tr} A)+g(ae_1+v,e_1)=-\operatorname{tr} A,
\end{align*}
yielding \eqref{Lee_avA}.
\end{proof}

\begin{remark}
Recall that a complex structure $J$ on a Lie algebra $\mathfrak{g}$ is called \emph{abelian} if it satisfies $[JX,JY]=[X,Y]$ for all $X,Y \in \mathfrak{g}$. Abelian complex structures occur very rarely on almost abelian Lie algebras. Let $J$ be an abelian complex structure on a $2n$-dimensional almost abelian Lie algebra $\mathfrak{g}$ and let $g$ be any $J$-Hermitian metric on $\mathfrak{g}$, so that we can consider an adapted unitary basis $\{e_1,\ldots,e_{2n}\}$ with associated algebraic data $(a,v,A)$. Then, using that $J$ is abelian and that $\mathfrak{n}$ is an abelian ideal, we have
\[
\text{ad}_{e_{2n}} \rvert_{\mathfrak{n}_1} = - \text{ad}_{e_1} \rvert_{\mathfrak{n}_1} \circ J_1 = 0,
\]
so that $A$ in \eqref{B} vanishes. Then, it is easy to see that $\mathfrak{g}$ is either isomorphic to $\R^{2n}$ (if $a=0$, $v=0$), the nilpotent $\mathfrak{h}_3 \oplus \R^{2n-3}$ (if $a=0$, $v \neq 0$) or the non-unimodular $\mathfrak{aff}_2 \oplus \R^{2n-2}$ (if $a \neq 0$), where $\mathfrak{h}_3\cong(0,0,f^{12})$ denotes the $3$-dimensional real Heisenberg algebra and $\mathfrak{aff}_2\cong(0,f^{12})$ denotes the $2$-dimensional real affine Lie algebra.
\end{remark}

Let $(\mathfrak{g}(a,v,A),J,g)$ be a Hermitian almost abelian Lie algebra and denote by $\{e_1,\ldots,e_{2n}\}$ the fixed unitary basis adapted to the splitting $\mathfrak{g}=J\mathfrak{k} \oplus \mathfrak{n}_1 \oplus \mathfrak{k}$.
In \cite{AL}, the authors proved that the \emph{Bismut-Ricci form} $\rho^B$ of $(\mathfrak{g}(a,v,A),J,g)$ is given by
\begin{equation}\label{rhoB}
\rho^B=-(a^2-\tfrac{1}{2}a\operatorname{tr}A+\lVert v \rVert^2)\, e^1 \wedge e^{2n} - (A^tv)^{\flat} \wedge e^{2n},
\end{equation}
while, as proven in \cite{LV}, its Chern-Ricci form $\rho^C$ is given by
\begin{equation} \label{rhoC}
\rho^C=-a\left(a+\tfrac{1}{2}\operatorname{tr}A\right) e^1 \wedge e^{2n}.
\end{equation}
Any of the two previous formulas can easily be obtained from the other.
Given any $1$-form $\alpha \in \mathfrak{g}^*$, since
$d\alpha(X,e_{2n})=\alpha([e_{2n},X])$, $X \in \mathfrak{g}$, one has
\[
d\alpha=(\text{ad}_{e_{2n}}^t \alpha) \wedge e^{2n} = (a\alpha(e_1)+\alpha(v))\,e^1 \wedge e^{2n} + A^t(\alpha\rvert_{\mathfrak{n}_1}) \wedge e^{2n}.
\]
Recalling \eqref{Lee_avA}, we then have
\[
-dJ\theta=\left( -a\operatorname{tr} A + \lVert v \rVert^2 \right)e^1 \wedge e^{2n} + (A^tv)^\flat \wedge e^{2n},
\]
which, by \eqref{ChernBismut}, can be used to determine \eqref{rhoC} from \eqref{rhoB} and vice versa.

We end this section with the characterization of Hermitian almost abelian Lie algebras  of complex dimension $n$ admitting a closed $(n,0)$-form.
We recall the following basic linear algebra fact: if $A$ is an endomorphism of a real vector space $V$ which commutes with a complex structure $J$ on $V$, $A$ can be seen as an endomorphism of the complex vector space $(V,J)$, which we denote by $A_{\C}$, whose trace satisfies
\begin{equation} \label{trC}
\operatorname{tr} A_{\C}= \tfrac{1}{2} \operatorname{tr} A - \tfrac{i}{2} \operatorname{tr} JA.
\end{equation}

\begin{proposition} \label{closed(n,0)}
A Hermitian almost abelian Lie algebra $(\mathfrak{g}(a,v,A),J,g)$ of complex dimension $n$  admits a closed $(n,0)$-form if and only if $a+\operatorname{tr}A_{\C}=0$, that is,
\[
a+\tfrac{1}{2} \operatorname{tr} A=0,\quad \operatorname{tr} JA =0.
\]
If $\mathfrak{g}(a,v,A)$ is unimodular, one has
\[
a=0,\quad \operatorname{tr}A=0,\quad \operatorname{tr} JA=0.
\]
\end{proposition}
\begin{proof}

With respect to the fixed adapted unitary basis $\{e_1,\ldots,e_{2n}\}$, consider the basis of $(1,0)$-forms $\{\alpha^1,\ldots,\alpha^n\}$ defined by $\alpha^k=e^k + i e^{2n+1-k}$, $k=1,\ldots,n$, together with their conjugate $(0,1)$-forms $\alpha^{\overline{k}}=e^k - i e^{2n+1-k}$, $k=1,\ldots,n$.

Since $d\mathfrak{g}^* \subset \mathfrak{g}^* \wedge e^{2n}$, the image of $\overline \partial \colon (\mathfrak{g}^{1,0})^* \to \Lambda^{1,1}\mathfrak{g}^*$ lies in $(\mathfrak{g}^{1,0})^* \wedge \alpha^{\overline{1}}$. Thus, $\overline \partial\rvert_{(\mathfrak{g}^{1,0})^*}$ can be identified with an endomorphism of $(\mathfrak{g}^{1,0})^*$. An explicit computation shows that its associated $n \times n$ matrix in the fixed basis $\{\alpha^1,\ldots,\alpha^n\}$ is given by
\[
\frac{i}{2}\begin{pmatrix} a & 0 \\ v_{\C} & A_{\C} \end{pmatrix},
\]
where $A_{\mathbb{C}}$ is seen as the complex $(n-1) \times (n-1)$ matrix corresponding to $A$ with respect to the complex basis 
$\left\{ z_k=\frac{1}{2}(e_k - i e_{2n+1-k} ) \right\}_{k=1,\ldots,n}$ of $\mathfrak{n}_1^{1,0}$ and $v_{\mathbb{C}}=\left(w_1,\ldots,w_{n-1}\right)^t \in \C^{n-1}$ such that
\[
v=\sum_{k=1}^{n-1} \left( w_k z_k + \overline{w_k} \overline{z_k} \right).
\]
We consider the $(n,0)$-form $\alpha^{1\ldots n}$ and compute
\begin{align*}
\overline\partial \alpha^{1\ldots n}
=\overline\partial \alpha^1 \wedge \alpha^{2\ldots n} - \sum_{k=2}^n (-1)^k \alpha^{1\ldots k-1} \wedge \overline\partial \alpha^k \wedge \alpha^{k+1 \ldots n}
=-\tfrac{i}{2}(a+\operatorname{tr} A_{\mathbb{C}}) \alpha^{\overline{1}1\ldots n},
\end{align*}
where, for example, $\alpha^{1\ldots n}$ is a shorthand for the wedge product $\alpha^1 \wedge \ldots \wedge \alpha^n$.
With the aid of \eqref{trC}, the first part of the claim follows. The second part easily follows by recalling that the unimodular hypothesis is equivalent to $a+\operatorname{tr}A=0$.
\end{proof}

Recall that a Hermitian structure $(J,g)$ on a smooth manifold $M$ is called \emph{locally conformally K\"ahler} (LCK) if $M$ admits an open cover $\{U_i\}_{i \in I}$ and smooth functions $f_i \in C^{\infty}(U_i)$, $i \in I$, such that $(J|_{U_i},e^{-f_i}g\rvert_{U_i})$ is a K\"ahler structure on $U_i$, for all $i \in I$. Equivalently, there exists a closed $1$-form $\alpha$ such that $d\omega=\alpha \wedge \omega$.

LCK almost abelian Lie algebras were characterized in \cite{AO}: a Hermitian almost abelian Lie algebra $(\mathfrak{g}(a,v,A),J,g)$ of real dimension $2n$ is LCK if and only if
\begin{equation} \label{LCK1}
	n=2,\quad A=0,
\end{equation}
yielding $\mathfrak{g}\cong \mathfrak{h}_3 \oplus \R$ (if $a=0$) or $\mathfrak{g} \cong \mathfrak{aff}_2 \oplus \R^2$ (if $a \neq 0$), or
\begin{equation} \label{LCK2}
	v=0,\quad A=\lambda\,\text{Id}_{\mathfrak{n}_1} + U,\;\,\lambda \in \R,\,U \in \mathfrak{u}(\mathfrak{n}_1,J_1,g),
\end{equation}
where $\mathfrak{u}(\mathfrak{n}_1,J_1,g)=\mathfrak{so}(\mathfrak{n}_1,g) \cap \mathfrak{gl}(\mathfrak{n}_1,J_1)$.

SKT almost abelian Lie algebras were investigated in \cite{AL}: a Hermitian almost abelian Lie algebra $(\mathfrak{g}(a,v,A),J,g)$ of real dimension $2n$ is SKT if and only if $[A,A^t]=0$ and the eigenvalues of $A$ have real part equal to $0$ or $-\frac{a}{2}$. A classification of  six-dimensional almost abelian Lie algebras admitting SKT structures was obtained in \cite{FP} and, recently, it was generalized to a wider class of $2$-step solvable Lie algebras in \cite{FS}.

\begin{proposition} \label{LCK_SKT}
A (non-K\"ahler) LCK almost abelian Lie algebra $(\mathfrak{g}(a,v,A),J,g)$ of complex dimension $n$ admitting a closed $(n,0)$-form is SKT if and only if $n=3$ or $n=2$ and $\mathfrak{g} \cong \mathfrak{h}_3 \oplus \R$. If $\mathfrak{g}$ is unimodular, then necessarily $\mathfrak{g} \cong \mathfrak{h}_3 \oplus \R$.
\end{proposition}

\begin{proof}
We divide the discussion depending on whether \eqref{LCK1} or \eqref{LCK2} holds. Assume \eqref{LCK2} holds. In particular, we note that $A$ satisfies $[A,A^t]=0$ and its eigenvalues have real part equal to $\lambda$. Necessarily we have $\lambda=\frac{\operatorname{tr} A}{2(n-1)}$. Since we are assuming there exists a closed $(n,0)$-form, by Proposition \ref{closed(n,0)} we have $a+ \frac{1}{2} \operatorname{tr} A=0$, which implies $\lambda=-\frac{a}{n-1}$.

Therefore, by the aforementioned characterization of SKT almost abelian Lie algebras, we obtain that $(\mathfrak{g}(a,v,A),J,g)$ is SKT if and only if $n=3$, for which we have $\lambda=-\frac{a}{2}$, or if $a=0$, which implies that the structure is K\"ahler, in particular. We also note that, by Proposition \ref{closed(n,0)}, if $\mathfrak{g}$ is unimodular, one obtains again $a=0$, implying the K\"ahler condition.

Now assume \eqref{LCK1} holds. It is easy to see that these conditions also imply that the Hermitian almost abelian Lie algebra is SKT (by \cite{AL}) and admits a closed $(2,0)$-form (by Proposition \ref{closed(n,0)}) if and only if $a=0$, concluding the proof.
\end{proof}

\section{Balanced metrics} \label{balancedmetrics}

Recall that a Hermitian metric is called \emph{balanced} if and only if $d^*\omega=0$, or equivalently if its Lee form  $\theta$ vanishes.

\begin{theorem} \label{Bal_avA}
A Hermitian almost abelian Lie algebra $(\mathfrak{g}(a,v,A),J,g)$  is balanced if and only if $v=0$, $\operatorname{tr} A=0$. For a balanced almost abelian Lie algebra $(\mathfrak{g}(a,v,A),J,g)$ the Bismut-Ricci form  is given  by 
\begin{equation} \label{rho_balanced}
\rho^B= \rho^C= -a^2\,e^1 \wedge e^{2n}.
\end{equation}
In particular,  a  unimodular   almost abelian  Lie algebra  admitting balanced metrics  is decomposable and its Bismut-Ricci form   $\rho^B$ vanishes.
\end{theorem}
\begin{proof} First, note that a Hermitian almost abelian Lie algebra $(\mathfrak{g}(a,v,A),J,g)$ satisfying $a=0$, $v=0$ is decomposable as the direct sum of two Lie algebras $J\mathfrak{k}$ and $\mathfrak{n}_1 \oplus \mathfrak{k}$.
The first part of the claim readily follows from Corollary \ref{leeform}. Together with \eqref{rhoB} and \eqref{rhoC}, this implies \eqref{rho_balanced}.

Now, if $\mathfrak{g}$ is unimodular, we also have 
\[
0=\operatorname{tr} \text{ad}_{e_{2n}}=a+\operatorname{tr} A=a,
\]
from which the final part of the claim follows.
\end{proof}

\begin{remark} \label{trAcondition}
Note that $\mathfrak{n}_1$ does not depend on $g$ and, for any $X_1,X_2 \in \mathfrak{g}$ transverse to $\mathfrak{n}$, the two endomorphisms $\text{ad}_{X_i}\rvert_{\mathfrak{n}_1}$, $i=1,2$, differ only by a multiplicative constant, since one can uniquely decompose $X_2=cX_1+Y$, $c \neq 0$, $Y \in \mathfrak{n}$, so that $\text{ad}_{X_2}\rvert_{\mathfrak{n}_1}=c\,\text{ad}_{X_1}\rvert_{\mathfrak{n}_1}$. In this way, it is clear that the condition $\operatorname{tr} A =0$ does not depend on the metric $g$ but only on the complex structure $J$, since it is equivalent to $\operatorname{tr} \text{ad}_{X}\rvert_{\mathfrak{n}_1}=0$ for any $X \in \mathfrak{g}$.
\end{remark}

Along the same lines, we can prove the following proposition, which is the exact analogous of \cite[Proposition 6.1]{FPS} in the almost abelian setting.

\begin{proposition} \label{BismutRicci_bal}
The Bismut connection of a Hermitian structure  $(J, g)$ on a unimodular almost abelian Lie algebra $\mathfrak{g}$ has holonomy contained in $\text{\normalfont SU}(n)$ if and only if  the Hermitian metric $g$  is balanced.
\end{proposition}
\begin{proof}

Let us fix an adapted unitary basis $\{e_1,\ldots,e_{2n}\}$ for $(\mathfrak{g},J,g)$, determining the algebraic data $(a,v,A)$ by \eqref{B}, with $a+\operatorname{tr} A=0$.
By \eqref{rhoB}, the Bismut-Ricci form reduces to
\[
\rho^B=-\left(\tfrac{3}{2}(\operatorname{tr} A)^2+\lVert v \rVert^2 \right) e^1 \wedge e^{2n} - \left(A^tv\right)^\flat \wedge e^{2n},
\]
Thus, its vanishing is equivalent to $a=\operatorname{tr}A=v=0$. By Theorem \ref{Bal_avA}, these conditions correspond precisely to the balanced condition in the unimodular case.
	
Moreover, we recall that, by \cite[Theorem 2]{WYZ}, a non-K\"ahler balanced Hermitian structure is never Bismut flat.
\end{proof}
\begin{remark} Notice that the previous proposition is not true in the non-unimodular case: for example, the six-dimensional non-balanced Hermitian almost abelian Lie algebra $(\mathfrak{g}(a,v,A),J,g)$ determined by the algebraic data
\[
a=1,\quad v=\left( \begin{smallmatrix} 0 \\ 1  \\ 0 \\ 0 \end{smallmatrix} \right), \quad A=\left( \begin{smallmatrix} 2 &0&0&0 \\ 0&0&0&0 \\ 0&0&0&0 \\ 0&0&0&2 \end{smallmatrix} \right),
\]
with respect to a unitary basis $\{e_1,\ldots,e_6\}$ adapted to the splitting $\mathfrak{g}=J\mathfrak{k} \oplus \mathfrak{n}_1 \oplus \mathfrak{k}$, $Je_i=e_{7-i}$, $i=1,2,3$,
has vanishing Bismut-Ricci form (and it is not Bismut flat), as shown by a direct computation.
\end{remark}

\begin{remark} \label{exact}
Special balanced structures are those for which $\omega^{n-1}$ is $d$-exact, or even $\partial \overline \partial$-exact. These stronger conditions can never be satisfied by a Hermitian structure on an almost abelian Lie algebra. In the notation we have introduced, the exactness of $\omega^{n-1}$ would imply $\omega^{n-1} \in \Lambda^{2n-3}\mathfrak{n}^* \wedge \mathfrak{k}^*$, so that $\omega^{n-1}\rvert_{\mathfrak{n}_1}=0$, which contradicts the non-degeneracy of $\omega\rvert_{\mathfrak{n}_1}$. 

Instead, as shown in \cite{Yac}, every Hermitian metric on a unimodular complex Lie algebra is such that $\omega^{n-1}$ is $\partial \overline \partial$-exact.
\end{remark}

It was conjectured in \cite{FV} that a compact complex manifold admitting both an SKT metric and a balanced metric necessarily admits a K\"ahler metric. 
We show that the conjecture holds for compact almost abelian solvmanifolds endowed with a left-invariant complex structure.
\begin{theorem} 
Let $\mathfrak{g}$ be an almost abelian Lie algebra endowed with a complex structure $J$. Assume $(\mathfrak{g},J)$ admits an SKT metric and a balanced metric. Then it admits a K\"ahler metric as well.
\end{theorem}
\begin{proof}
Denote the SKT metric by $g$ and the balanced metric by $g^\prime$.
Fix a $(J,g)$-unitary adapted basis $\{e_1,\ldots,e_{2n}\}$, so that, by \cite{AL}, the corresponding data $(a,v,A)$ in \eqref{B} is such that $[A,A^t]=0$ and the real part of the eigenvalues of $A$ is equal to either $-\frac{a}{2}$ or $0$. The existence of a $J$-Hermitian balanced metric now forces $\operatorname{tr} A=0$ (recall Remark \ref{trAcondition}), so by \cite[Corollary 4.6]{AL} $A$ should be skew-symmetric with respect to the metric $g$.

Now, let $\{e_1^\prime,\ldots,e_{2n}^\prime\}$ be a $(J,g^\prime)$-unitary adapted basis, with associated data $(a^\prime,v^\prime,A^\prime)$. By Theorem \ref{Bal_avA}, we have in particular $v^\prime=0$, which means that $\text{ad}_{e_{2n}^\prime}$ preserves the line generated by $e_1^\prime=-Je_{2n}^\prime$.

Consider now the metric
\[
g^{\prime \prime}=g\rvert_{\mathfrak{n}_1} + g^\prime\rvert_{\text{span}\left<e_1^\prime,e_{2n}^\prime\right>}.
\]
Then $g^{\prime \prime}$ is $J$-Hermitian since it is the orthogonal sum of two $J$-Hermitian metrics on two $J$-invariant subspaces. We have that $\{e_1^\prime,e_2,\ldots,e_{2n-1},e_{2n}^\prime\}$ is a $(J,g^{\prime \prime})$-unitary adapted basis such that the matrix $B^{\prime \prime}$ associated with $\text{ad}_{e_{2n}^\prime}\rvert_{\mathfrak{n}}$ is of the form
\[
B^{\prime \prime}=\begin{pmatrix} a^\prime & 0 \\ 0 & A^{\prime \prime} \end{pmatrix},
\]
where $A^{\prime \prime}$ is a multiple of $A$ (by Remark \ref{trAcondition}), hence skew-symmetric with respect to $g$ and $g^{\prime \prime}$, since the two metrics coincide on $\mathfrak{n}_1$. From the characterization of K\"ahler almost abelian Lie algebras (see \cite{LW}) it follows that $g^{\prime \prime}$ is a K\"ahler metric.
\end{proof}

Recall that, by the symmetrization process described in \cite{Bel, FG, Uga}, the existence of an SKT (resp.\ balanced) metric on a compact complex solvmanifold $(\Gamma \backslash G,J)$ implies the existence of an invariant SKT (resp.\ balanced) metric on $(\Gamma \backslash G,J)$.

\begin{corollary}
Let $\Gamma \backslash G$ be an almost abelian solvmanifold endowed with a left-invariant complex structure $J$. Assume $(\Gamma \backslash G,J)$ admits an SKT metric and a balanced metric. Then it admits a K\"ahler metric as well.
\end{corollary}

We note that there are plenty of compact K\"ahler solvmanifolds. By \cite{Has}, these are all finite quotients of complex tori which have the structure of a complex torus bundle over a complex torus.

\section{Classification in  dimension six} \label{6-dim}

Six-dimensional nilpotent Lie algebras carrying balanced structures were classified in \cite{Uga}: among them, apart from the abelian Lie algebra $\R^6$, the only one which is almost abelian has structure equations
\begin{equation} \label{nilp_almostabelian}
(0,0,0,0,f^{12},f^{13}),
\end{equation}
up to isomorphism, where the notation we adopt means that the Lie algebra \eqref{nilp_almostabelian} admits a basis $\{f_1,\ldots,f_6\}$ whose dual basis $\{f^1,\ldots,f^6\}$ satisfies $df^1=df^2=df^3=df^4=0$, $df^5=f^{12}$, $df^6=f^{13}$.

In the next theorem, we classify six-dimensional  (non-K\"ahler) non-nilpotent almost abelian Lie  algebras  admitting a balanced structure.
Six-dimensional non-nilpotent almost abelian Lie algebras admitting a complex structure were classified in \cite[Theorem 3.2]{FP} and can be grouped in classes of isomorphism $\mathfrak{k}_i$, $i=1,\ldots,26$, each depending on some parameters.

In what follows, we have decided to use a less cumbersome notation, with respect to \cite{FP}, for the Lie algebras among that list which admit balanced structures. In two instances, two different families of isomorphism classes in \cite[Theorem 3.2]{FP} have been grouped together, by allowing the parameters to take a wider set of values. As in \cite{FP}, the resulting list is not redundant

In Table \ref{table_cpx}, we also provide explicit examples of balanced structures on those Lie algebras, highlighting the aforementioned correspondence with the list in \cite{FP}.

\begin{theorem} \label{Bal_class}
Let $\mathfrak{g}$ be a six-dimensional non-K\"ahler non-nilpotent almost abelian Lie algebra. Then $\mathfrak{g}$ admits a balanced structure $(J,g)$ if and only if it is isomorphic to one of the following:
\begin{itemize}
\setlength{\itemindent}{-1em}
\item[] $\mathfrak{b}_1=(f^{16},pf^{26},pf^{36},-pf^{46},-pf^{56},0)$,  \,  $p \neq 0$,\smallskip
\item[] $\mathfrak{b}_{2}=(f^{16},f^{36},0,f^{56},0,0)$,\smallskip
\item[] $\mathfrak{b}_{3}=(pf^{16},qf^{26},qf^{36},-qf^{46}+f^{56},-f^{46}-qf^{56},0)$,  \,  $pq \neq 0$, \smallskip
\item[] $\mathfrak{b}_{4}=(pf^{16},qf^{26}+f^{36},-f^{26}+qf^{36},-qf^{46}+sf^{56},-sf^{46}-qf^{56},0)$, \, $pqs \neq 0$,\smallskip
\item[] $\mathfrak{b}_{5}=(pf^{16},f^{36}-f^{46},-f^{26}-f^{56},f^{56},-f^{46},0)$, \,  $p \neq 0$, \smallskip
\item[] $\mathfrak{b}_{6}=(f^{16},f^{26},-f^{36},-f^{46},0,0)$, \smallskip
\item[] $\mathfrak{b}_{7}=(f^{16},f^{26},-f^{36}+rf^{46},-rf^{36}-f^{46},0,0)$, \, $r \neq 0$, \smallskip
\item[] $\mathfrak{b}_{8}=(pf^{16}+f^{26},-f^{16}+pf^{26},-pf^{36}+rf^{46},-rf^{36}-pf^{46},0,0)$, \, $pr \neq 0$, \smallskip
\item[] $\mathfrak{b}_{9}=(f^{26}-f^{36},-f^{16}-f^{46},f^{46},-f^{36},0,0)$.
\end{itemize} 
Among these, only $\mathfrak{b}_{6}$, $\mathfrak{b}_{7}$, $\mathfrak{b}_{8}$ and $\mathfrak{b}_{9}$ are unimodular.
\end{theorem}
\begin{proof}
Let $(J,g)$ be a  balanced structure on $\mathfrak{g}$. Let $\{e_1,\ldots,e_6\}$ be a unitary basis of $(\mathfrak{g},J,g)$ adapted to the splitting $\mathfrak{g}=J\mathfrak{k} \oplus \mathfrak{n}_1 \oplus \mathfrak{k}$, so that the matrix $B$ associated with $\text{ad}_{e_6}\rvert_{\mathfrak{n}}$ is of the form \eqref{B}, with the additional conditions $v=0$, $\operatorname{tr} A=0$, by Theorem \ref{Bal_avA}, i.e.,
\begin{equation} \label{B_1}
B= \begin{pmatrix} a & 0 \\ 0 & A \end{pmatrix}, \quad a \in \R,\,A \in \mathfrak{gl}(\mathfrak{n}_1,J_1),\, \operatorname{tr}A=0.
\end{equation}
Now, following the approach of the proof of \cite[Theorem 3.2]{FP}, the crucial step in order to obtain a nice representative for the isomorphism class of $\mathfrak{g}$ is to obtain a canonical matrix form for the endomorphism $A$. By focusing on the condition $[A,J_1]=0$ for the integrability of $J$ and using the results in the aforementioned proof, one obtains that there must exist a new basis $\{e_2^\prime,e_3^\prime,e_4^\prime,e_5^\prime\}$ for the $J$-invariant subspace $\mathfrak{n}_1$ and possibly a rescaling of $e_6$ such that the endomorphism $\text{ad}_{e_6}\rvert_{\mathfrak{n}_1}$ is now represented by one of the following matrices
\[
	A_1\!=\!\!\left( \begin{smallmatrix} p&0&0&0 \\ 0&p&0&0 \\ 0&0&q&0 \\ 0&0&0&q \end{smallmatrix} \right)\!\!, \;
	A_2\!=\!\!\left( \begin{smallmatrix} p&1&0&0 \\ -1&p&0&0 \\ 0&0&q&0 \\ 0&0&0&q \end{smallmatrix} \right)\!\!, \;
	A_3\!=\!\!\left( \begin{smallmatrix} p&1&0&0 \\ -1&p&0&0 \\ 0&0&q&r \\ 0&0&-r&q \end{smallmatrix} \right)\!\!, \;
	A_4\!=\!\!\left( \begin{smallmatrix} p&1&0&0 \\ 0&p&0&0 \\ 0&0&p&1 \\ 0&0&0&p \end{smallmatrix} \right)\!\!, \;
	A_5\!=\!\!\left( \begin{smallmatrix} p&1&-1&0 \\ -1&p&0&-1 \\ 0&0&p&1 \\ 0&0&-1&p \end{smallmatrix} \right)\!\!,
\]
for some $p,q,r \in \R$, $r \neq 0$ to avoid $A_2$ being a subcase of $A_3$. We still have to impose the traceless condition for $A$, which easily translate in the traceless condition for the matrices $A_1,\ldots,A_5$, since the condition does not change after changes of basis for $\mathfrak{n}_1$ and uniform rescalings. We thus obtain $p=q$ in $A_1$, $q=-p$ in $A_2$ and $A_3$, $p=0$ in $A_4$ and $A_5$. Now, depending on $a$ in \eqref{B_1} and $p$ being zero or non-zero, it is easy to see which of the algebras of \cite[Theorem 3.2]{FP} admit balanced structures. Removing the algebras admitting K\"ahler structures (which were classified in \cite[Theorem 3.8]{FP}) and the nilpotent Lie algebra \eqref{nilp_almostabelian} (which can be obtained from $A_4$, when $a=0$), we have that
\begin{itemize}
\item[] $A_1$ yields $\mathfrak{b}_1$ and $\mathfrak{b}_{6}$,
\item[] $A_2$ yields $\mathfrak{b}_{3}$ and $\mathfrak{b}_{7}$,
\item[] $A_3$ yields $\mathfrak{b}_{4}$ and $\mathfrak{b}_{8}$,
\item[] $A_4$ yields $\mathfrak{b}_{2}$,
\item[] $A_5$ yields $\mathfrak{b}_{5}$ and $\mathfrak{b}_{9}$,
\end{itemize}
from which the claim follows.

The list of eleven Lie algebras in \cite[Theorem 3.2]{FP} reduces to nine, in our notation, since the Lie algebra $\mathfrak{b}_1$ encompasses $\mathfrak{k}_1^{p,-p}$ and $\mathfrak{k}_3^{-1,\frac{1}{p}}$, while $\mathfrak{b}_3$ corresponds to $\mathfrak{k}_8^{p,q,-q}$ or $\mathfrak{k}_9^{q,p,-q}$.
\end{proof}

\begin{table}
\begin{center}
\addtolength{\leftskip} {-2cm}
\addtolength{\rightskip}{-2cm}
\scalebox{0.75}{
\begin{tabular}{|l|l|l|l|l|}
\hline \xrowht{15pt}
Name & Structure equations & Isomorphism with \cite[Theorem 3.2] {FP} & Unimodular & Balanced structure \\
\hline \hline
\xrowht{20pt}
$\mathfrak{b}_1$ & $(f^{16},pf^{26},pf^{36},-pf^{46},-pf^{56},0)$, $p \neq 0$ & \begin{tabular}{@{}l@{}} $\mathfrak{k}_1^{p,-p}$, if $|p| \geq 1$, or \\ $\mathfrak{k}_3^{-1,\frac{1}{p}}$, if $|p| <1$ \\ \end{tabular} & \xmark & \begin{tabular}{@{}l@{}} $Jf_1=f_6$, $Jf_2=f_3$, $Jf_4=f_5$ \\ $g=\sum_{k=1}^6 (f^k)^2$ \\ \end{tabular}  \\ \hline
\xrowht{20pt}
$\mathfrak{b}_2$ & $(f^{16},f^{36},0,f^{56},0,0)$ & $\mathfrak{k}_6^0$ & \xmark & \begin{tabular}{@{}l@{}} $Jf_1=f_6$, $Jf_2=f_4$, $Jf_3=f_5$ \\ $g=\sum_{k=1}^6 (f^k)^2$ \\ \end{tabular} \\ \hline
\xrowht{20pt}
$\mathfrak{b}_3$ & $(pf^{16},qf^{26},qf^{36},-qf^{46}+f^{56},-f^{46}-qf^{56},0)$, $pq \neq 0$ & \begin{tabular}{@{}l@{}} $\mathfrak{k}_8^{p,q,-q}$, if $|p| \geq |q|$, or \\ $\mathfrak{k}_9^{q,p,-q}$, if $|p| < |q|$ \\ \end{tabular} & \xmark & \begin{tabular}{@{}l@{}} $Jf_1=f_6$, $Jf_2=f_3$, $Jf_4=f_5$ \\ $g=\sum_{k=1}^6 (f^k)^2$ \\ \end{tabular} \\ \hline
\xrowht{20pt}
$\mathfrak{b}_4$ & $(pf^{16},qf^{26}+f^{36},-f^{26}+qf^{36},-qf^{46}+sf^{56},-sf^{46}-qf^{56},0)$, $pqs \neq 0$ & $\mathfrak{k}_{11}^{p,q,-q,s}$ & \xmark & \begin{tabular}{@{}l@{}} $Jf_1=f_6$, $Jf_2=f_3$, $Jf_4=f_5$ \\ $g=\sum_{k=1}^6 (f^k)^2$ \\ \end{tabular} \\ \hline
\xrowht{20pt}
$\mathfrak{b}_5$ & $(pf^{16},f^{36}-f^{46},-f^{26}-f^{56},f^{56},-f^{46},0)$, $p \neq 0$ & $\mathfrak{k}_{12}^{p,0}$ & \xmark & \begin{tabular}{@{}l@{}} $Jf_1=f_6$, $Jf_2=f_3$, $Jf_4=f_5$ \\ $g=\sum_{k=1}^6 (f^k)^2$ \\ \end{tabular} \\ \hline
\xrowht{20pt}
$\mathfrak{b}_6$ & $(f^{16},f^{26},-f^{36},-f^{46},0,0)$ & $\mathfrak{k}_{20}^{-1}$ & \cmark & \begin{tabular}{@{}l@{}} $Jf_1=f_2$, $Jf_3=f_4$, $Jf_5=f_6$ \\ $g=\sum_{k=1}^6 (f^k)^2$ \\ \end{tabular} \\ \hline
\xrowht{20pt}
$\mathfrak{b}_7$ & $(f^{16},f^{26},-f^{36}+rf^{46},-rf^{36}-f^{46},0,0)$, $r \neq 0$ & $\mathfrak{k}_{22}^{-1,r}$ & \cmark & \begin{tabular}{@{}l@{}} $Jf_1=f_2$, $Jf_3=f_4$, $Jf_5=f_6$ \\ $g=\sum_{k=1}^6 (f^k)^2$ \\ \end{tabular} \\ \hline
\xrowht{20pt}
$\mathfrak{b}_8$ & $(pf^{16}+f^{26},-f^{16}+pf^{26},-pf^{36}+rf^{46},-rf^{36}-pf^{46},0,0)$, $pr \neq 0$ & $\mathfrak{k}_{25}^{p,-p,r}$ & \cmark & \begin{tabular}{@{}l@{}} $Jf_1=f_2$, $Jf_3=f_4$, $Jf_5=f_6$ \\ $g=\sum_{k=1}^6 (f^k)^2$ \\ \end{tabular} \\ \hline
\xrowht{20pt}
$\mathfrak{b}_9$ & $(f^{26}-f^{36},-f^{16}-f^{46},f^{46},-f^{36},0,0)$ & $\mathfrak{k}_{26}^0$ & \cmark & \begin{tabular}{@{}l@{}} $Jf_1=f_2$, $Jf_3=f_4$, $Jf_5=f_6$ \\ $g=\sum_{k=1}^6 (f^k)^2$ \\ \end{tabular} \\ \hline
\end{tabular}}
\smallskip   
\caption{Six-dimensional (non-K\"ahler) non-nilpotent almost abelian Lie algebras admitting balanced structures.} \label{table_cpx}
\end{center}
\end{table}

\begin{remark}
We observe that some redundancy is present in the list of Lie algebras of Theorem \ref{Bal_class}, namely, within the same family $\mathfrak{b}_k$, different choices of parameters can yield isomorphic Lie algebras. For example, this happens in $\mathfrak{b}_1$ when exchanging $p$ with $-p$.

It is never impossible, though, to have $\mathfrak{b}_k$ isomorphic to $\mathfrak{b}_{k^\prime}$, for $k \neq k^\prime$, for any choice of the parameters involved.
\end{remark}

\begin{remark} \label{lattices_bal}
By \cite[Proposition 7.2.1]{Bock}, the Lie group corresponding to the Lie algebra $\mathfrak{b}_{6}$ \cm{$\mathfrak{k}_{20}^{-1}$} admits compact quotients by lattices. By \cite{Har} (see also \cite{Bock, CM}), the same is true for the Lie groups corresponding to the Lie algebras $\mathfrak{b}_{8}$ \cm{$\mathfrak{k}_{25}^{p,-p,r}$} (for some values of the parameters $p,r$) and $\mathfrak{b}_{9}$\cm{$\mathfrak{k}_{26}^0$}. As we shall now show, none of the Lie groups with Lie algebra  $\mathfrak{b}_{7}$ \cm{$\mathfrak{k}_{22}^{-1,r}$} admit compact quotients by lattices: following \cite{Bock}, a co-compact lattice exists on such Lie groups if and only if there exists a non-zero $t_0 \in \R$ and a basis of $\mathfrak{n}$ such that the matrix associated with $\text{exp}(t_0\text{ad}_{f_6})\rvert_{\mathfrak{n}}$ has integer entries. In the basis $\{f_1,\ldots,f_5\}$, one computes
\begin{equation} \label{exptad}
\text{exp}(t\,\text{ad}_{f_6})\rvert_{\mathfrak{n}}=
\left(\begin{smallmatrix}
e^t & 0 & 0 & 0 & 0 \\
0 & e^t & 0 & 0 & 0 \\
0 & 0 & e^{-t}\cos(tr) & e^{-t} \sin(tr) & 0 \\
0 & 0 & -e^{-t}\sin(tr) & e^{-t} \cos(tr) & 0 \\
0 & 0 & 0 & 0 & 1
\end{smallmatrix}\right),\quad t \in \R,
\end{equation}
having minimal polynomial $p_t(\lambda)=\sum_{i=0}^4 a_i(t,r) \lambda^i$, with
\begin{gather*}
a_0=e^{-t},\quad a_1=- 2\cos(tr)-e^{-t}-e^{-2t},\quad a_2=2\cos(tr)+e^{t} +2\cos(tr)e^{-t}+e^{-2t}, \\a_3= - 1-2e^{-t}\cos(tr)-e^t,\quad a_4=1.
\end{gather*}
If \eqref{exptad} is conjugate to an integer matrix for some $t_0$, then necessarily $p_{t_0}(\lambda)$ is an integer polynomial, so that $a_0(t_0,r) \in \mathbb{Z}$ forces $t_0=-\log(k)$, for some $k \in \mathbb{Z}_{>0}$. Together with $a_1(t_0,r) \in \mathbb{Z}$, this implies $2\cos(r\log(k))=h \in \mathbb{Z}$, and now $a_3(t_0,r)=-hk-1-\frac{1}{k}$ is integer if and only if $k=1$, that is, $t_0=0$, a contradiction.
\end{remark}

\begin{example}
Recall Proposition \ref{BismutRicci_bal}. By Theorem \ref{Bal_class} and Remark \ref{lattices_bal}, we have that the simply connected Lie groups with Lie algebra $\mathfrak{b}_{6}$\cm{$\mathfrak{k}_{20}^{-1}$}, $\mathfrak{b}_{8}$\cm{$\mathfrak{k}_{25}^{p,-p,r}$}, and $\mathfrak{b}_{9}$ \cm{$\mathfrak{k}_{26}^0$} have compact quotients by lattices admitting left-invariant Hermitian structures satisfying $\text{Hol}^0(\nabla^B) \subseteq \text{SU}(3)$. In particular, we now show that they admit balanced Hermitian structures such that $\text{Hol}^0(\nabla^B)=\text{SU}(3)$.

Consider the Hermitian structure $(g,J)$ defined by
\[
g=\sum_{i=1}^6 (f^i)^2,\quad Jf_1=f_2,\,Jf_3=f_4,\,Jf_5=f_6.
\]
Consider the curvature $2$-forms $\Omega^i_j$ associated with the Bismut connection $\nabla^B$ of $(g,J)$, $R^{B}=\Omega^i_j \, f^j \otimes f_i$. On all three algebras, these span an $8$-dimensional space, proving our claim (see \cite{DFFU}): in particular, a set of eight linearly independent curvature $2$-forms is provided by
\begin{alignat*}{4}
\Omega^1_2&=-2f^{12}, \quad&
\Omega^1_3&=f^{13}+f^{24}, \quad&
\Omega^1_4&=f^{14}-f^{23}, \quad&
\Omega^1_5&=-f^{15}-f^{26}, \\
\Omega^2_5&=f^{16}-f^{25}, \quad&
\Omega^3_4&=-2f^{34}, \quad&
\Omega^3_5&=-f^{35}-f^{46}, \quad&
\Omega^4_5&=f^{36}-f^{45},
\end{alignat*}
on $\mathfrak{b}_{6}$\cm{$\mathfrak{k}_{20}^{-1}$}, by
\begin{alignat*}{4}
\Omega^1_2&=-2p^2f^{12}, \quad&
\Omega^1_3&=p^2f^{13}+p^2f^{24}, \quad&
\Omega^1_4&=p^2f^{14}-p^2f^{23}, \quad&
\Omega^1_5&=-p^2f^{15}-p^2f^{26}, \\
\Omega^2_5&=p^2f^{16}-p^2f^{25}, \quad&
\Omega^3_4&=-2p^2f^{34}, \quad&
\Omega^3_5&=-p^2f^{35}-p^2f^{46}, \quad&
\Omega^4_5&=p^2f^{36}-p^2f^{45},
\end{alignat*}
on $\mathfrak{b}_{8}$ and by
\begin{alignat*}{4}
\Omega^1_2&=-\tfrac{1}{4}f^{34}-\tfrac{1}{4}f^{56}, \quad&
\Omega^1_3&=\tfrac{1}{4}f^{13}+\tfrac{1}{4}f^{24}, \quad&
\Omega^1_4&=-\tfrac{1}{4}f^{14}+\tfrac{1}{4}f^{23}, \quad&
\Omega^1_5&=-\tfrac{1}{4}f^{15}+\tfrac{1}{4}f^{26}, \\
\Omega^2_5&=-\tfrac{1}{4}f^{16}-\tfrac{1}{4}f^{25}, \quad&
\Omega^3_4&=-\tfrac{1}{2}f^{12}+\tfrac{1}{2}f^{56}, \quad&
\Omega^3_5&=-\tfrac{1}{4}f^{35}-\tfrac{3}{4}f^{46}, \quad&
\Omega^4_5&=\tfrac{3}{4}f^{36}-\tfrac{1}{4}f^{45},
\end{alignat*}
on $\mathfrak{b}_{9}$\cm{$\mathfrak{k}_{26}^0$}.

In addition, note that the $(3,0)$-form $\Psi=(f^1+if^2)\wedge(f^3+if^4)\wedge(f^5+if^6)$ is closed on $\mathfrak{b}_{6}$ and $\mathfrak{b}_{8}$ ($r=1$), so that the corresponding compact quotients can be endowed with a balanced $\text{\normalfont SU}(3)$-structure (recall the definition in \cite[Section 1]{FTUV}) satisfying $\text{Hol}^0(\nabla^B)=\text{SU}(3)$.
\end{example}

\section{Balanced flow} \label{sec_bal}

In \cite{BV}, the authors introduced a parabolic flow for Hermitian metrics on a complex manifold, preserving the balanced condition of the initial data: 
in terms of the $(n-1)$-st power of the fundamental form, the evolution equation reads
\begin{equation} \label{flowBV}
\tfrac{\partial}{\partial t}\varphi(t)=i \partial \overline\partial *_t \left(\rho_t^C \wedge *_t \varphi(t)\right) +\Delta_{BC} \varphi(t),\quad \varphi(0)=\varphi_0=*_0 \omega_0.
\end{equation}
In terms of $\omega(t)$ (which, hereafter, we identify with the evolving metric $g(t)$), equivalently we have
\begin{equation} \label{flowBV2}
\tfrac{\partial}{\partial t}\omega(t)=(n-2)!\,\iota_{\omega(t)^{n-2}}\left(i\partial \overline{\partial} *_t \left(\rho_t^C \wedge \omega(t) \right) \right) + \frac{1}{n-1} \, \iota_{\omega(t)^{n-2}} \Delta_{BC} \omega(t)^{n-1},\quad \omega(0)=\omega_0,
\end{equation}
where $*_t$ is the Hodge star operator associated with $\omega(t)$, $\rho^C_t$ is the Chern-Ricci form of $\omega(t)$ and $\Delta_{BC}$ is the modified Bott-Chern Laplacian
\[
\Delta_{B C}=\partial \overline{\partial \partial}^{*} \partial^{*}+\bar{\partial}^{*} \partial^{*} \partial \bar{\partial}+\bar{\partial}^{*} \partial \partial^{*} \bar{\partial}+\partial^{*} \overline{\partial \partial}^{*} \partial+\bar{\partial}^{*} \bar{\partial}+\partial^{*} \partial
\]
associated with $\omega(t)$. The flow \eqref{flowBV} differs from  (1.3) in \cite{BV} by a factor for the second summand in the right-hand side, as suggested by the authors \cite{BV2}. We shall refer to this flow as \emph{balanced flow} and we shall denote the right-hand side of \eqref{flowBV2} by $q(\omega(t))$ for simplicity.  As shown in \cite{BV}, \eqref{flowBV} admits a unique solution defined on some maximal interval $[0,\varepsilon)$, with $[\varphi(t)]=[\varphi_0] \in H_{BC}^{2n-2}={\ker d}/{\operatorname{Im} \partial \overline\partial}$ for all $t \in [0,\varepsilon)$. In particular, the balanced flow preserves the $\partial \overline \partial$-exactness of $\omega^{n-1}_0$ (see Remark \ref{exact}).

When the initial data is K\"ahler, the balanced flow reduces to the \emph{Calabi flow} (see \cite{Cal, Cal2}),
\begin{equation} \label{calabiflow}
	\tfrac{\partial}{\partial t} \omega(t)=i\partial \overline\partial s_{\omega(t)},\quad \omega(0)=\omega_0,
\end{equation}
$s_{\omega(t)}$ denoting the scalar curvature of $\omega(t)$, or equivalently
\begin{equation} \label{calabiflow2}
	\tfrac{\partial}{\partial t}\varphi(t)=i \partial \overline\partial *_t \left(\rho_t \wedge *_t \varphi(t)\right),\quad \varphi(0)=\varphi_0=*_0 \omega_0,
\end{equation}
$\rho_t$ being the Ricci form associated with $\omega(t)$.

Consider a simply connected  almost abelian Lie group $G$ endowed with a left-invariant Hermitian structure $(J,g_0,\omega_0)$ and let $\mathfrak{g}$ be its Lie algebra. If one assumes that the solution to the balanced flow remains left-invariant for all times, then \eqref{flowBV} and \eqref{flowBV2} reduce to systems of \textsc{ode}s on the Lie algebra $\mathfrak{g}$, which admit a unique solution. By the uniqueness of solutions of the balanced flow, one obtains that left-invariant initial data yield left-invariant solutions.

Since the scalar curvature of a left-invariant metric on a Lie group is constant, it is immediate to see that left-invariant K\"ahler metrics on Lie groups are fixed points of the Calabi flow and consequently of the balanced flow.

To study the balanced flow on almost abelian Lie groups in the non-K\"ahler case, we apply the \emph{bracket flow} technique introduced by Lauret in  \cite{Lau} and already successfully applied in the study of other Hermitian flows (see for example \cite{AL} for the case of the pluriclosed flow for SKT metrics).

On the vector space $\R^{2n}$, fix the complex structure $J_0$ given by $J_0e_i=e_{2n+1-i}$, $i=1,\ldots,n$, and standard scalar product $\left< \cdot, \cdot \right>=\sum_{i=1}^n (e^i)^2$, where $\{e_1,\ldots,e_{2n}\}$ denotes the standard basis of $\R^{2n}$.

Consider a $2n$-dimensional Hermitian Lie algebra $(\mathfrak{g},J,g)$. Fixing a unitary basis $\{e_1,\ldots,e_{2n}\}$, $Je_i=e_{2n+1-i}$, $i=1,\ldots,n$, for $(\mathfrak{g},J,g)$ is equivalent to identifying $(\mathfrak{g},J,g)$ with $(\R^{2n},J_0,\left< \cdot, \cdot \right>)$. The Hermitian structure is thus fully determined by the bracket operation $\mu \in V_{2n} = \Lambda^2 (\R^{2n})^* \otimes \R^{2n}$ corresponding to the Lie bracket of $\mathfrak{g}$ under this identification. We denote the corresponding Hermitian Lie algebra by $(\mu,J_0,\left< \cdot,\cdot \right>)$ and the induced simply connected Hermitian Lie group by $(G,\mu,J_0,\left< \cdot,\cdot \right>)$. When $\mathfrak{g}$ is almost abelian, we further ask for the unitary basis to be adapted to the splitting $\mathfrak{g}=J\mathfrak{k} \oplus \mathfrak{n}_1 \oplus \mathfrak{k}$ and we denote the corresponding bracket on $\R^{2n}$ by $\mu(a,v,A)$ to keep track of the algebraic data determined by \eqref{B}. We denote the induced splitting of $\R^{2n}$ with $J\mathfrak{k} \oplus \mathfrak{n}_1 \oplus \mathfrak{k}$ as well.

The Lie group $\text{GL}(2n,J_0)$ of automorphisms of $\R^{2n}$ preserving $J_0$ acts transitively on the set of $J_0$-Hermitian metrics via pullback, so that the balanced flow starting from a Hermitian structure $(\mu_0,J_0,\left< \cdot, \cdot \right>)$ yields a family $(\mu_0,J_0,h(t)^*\left< \cdot, \cdot \right>)$, for some $h(t) \subset \text{GL}(2n,J_0)$.

One then observes that
\[
h(t) \colon (\mu_0,J_0,h(t)^*\left< \cdot, \cdot \right>) \to (h(t) \cdot \mu_0,J_0,\left< \cdot, \cdot \right>)
\]
is an isomorphism of Hermitian structures, namely $h(t)$ is a Lie algebra isomorphism which is both orthogonal and biholomorphic. Here we denoted
\[
h \cdot \mu = (h^{-1})^*\mu=h\mu(h^{-1} \cdot, h^{-1} \cdot).
\]
Let $\mu(t)=h(t) \cdot \mu_0$. Then,  up to time-dependent biholomorphisms, the balanced flow starting from  $(\mu_0,J_0,\left< \cdot, \cdot \right>)$ can be interpreted as a flow $\mu(t)$ on $V_{2n}$, such that $\mu(t) \in \text{GL}(2n,J_0) \cdot \mu_0$ for all $t$. For any Lie bracket $\mu$,
denote by $q_{\mu}$ the restriction to $\Lambda^{2}T_0^*\R^{2n} \cong \Lambda^2(\R^{2n})^*$ of $q(\omega_\mu)$, where $\omega_\mu$ is the left-invariant extension of $\Omega_0=\left<J_0 \cdot, \cdot \right>$ on $\R^{2n}$ according to the Lie group operation corresponding to the bracket $\mu$ (see \cite[Section 2]{Lau0}).

The  evolution of
$\mu(t)$ is given by the so-called bracket flow
\begin{equation} \label{brflow}
\tfrac{d}{dt}\mu= - \pi(Q_\mu)\mu,\quad \mu(0)=\mu_0,
\end{equation}
where
\begin{equation} \label{Qmu}
Q_\mu = - \tfrac{1}{2} \Omega_0^{-1} q_{\mu} \in \mathfrak{gl}_{2n}
\end{equation}
and 
\[
(\pi(A)\mu)(X,Y)=A\mu(X,Y) - \mu(AX,Y) - \mu(X,AY),
\]
for any $A \in \mathfrak{gl}_{2n}$, $\mu \in V_{2n}$, $X,Y \in \R^{2n}$.

\begin{proposition} \label{propQmu}
For a balanced almost abelian Lie algebra $(\mu(a,0,A),J_0,\left< \cdot, \cdot \right>)$, $\operatorname{tr} A=0$, the endomorphism $Q_{\mu(a,0,A)}$ is given by
\begin{equation} \label{Qmu_almostabelian}
Q_{\mu(a,0,A)}=\begin{pmatrix} p & 0 & 0 \\ 0 & P & 0 \\ 0 & 0 & p \end{pmatrix},
\end{equation}
in terms of the fixed splitting $\mathfrak{g}=J\mathfrak{k} \oplus \mathfrak{n}_1 \oplus \mathfrak{k}$, where $p=p(a,A)$ is a homogeneous fourth order polynomial in $a$ and the entries of $A$ and $P=P(a,A)$ is a symmetric endomorphism of $\mathfrak{n}_1$ commuting with $J_1=J\rvert_{\mathfrak{n}_1}$ whose entries are fourth order polynomials in $a$ and the entries of $A$.
\end{proposition}
\begin{proof}
We start by showing that $q_{\mu}$ lies in $\mathfrak{k}^* \wedge J\mathfrak{k}^* \oplus \Lambda^2 \mathfrak{n}_1^*$. For simplicity, let us write $\omega$ in place of $\omega_\mu$. Notice first that the balanced condition $d\omega^{n-1}=0$ implies
\[
\Delta_{BC} \omega^{n-1}=-(n-1)!\,\partial \overline{\partial} * \partial \overline{\partial} \omega = \tfrac{(n-1)!}{4} dJd * dJd\omega.
\]
Then, using that $\omega \in \mathfrak{k}^* \wedge J\mathfrak{k}^* \oplus \Lambda^2 \mathfrak{n}_1^*$, that $d\mathfrak{g}^* \subset \mathfrak{k}^* \wedge J\mathfrak{k}^* \oplus \mathfrak{k}^* \wedge \mathfrak{n}_1^*$ (which is a consequence of $v=0$) and that $*(\mathfrak{k}^* \wedge J\mathfrak{k}^* \wedge \Lambda^2 \mathfrak{n}_1^*) \subset \Lambda^{2n-4}\mathfrak{n}_1^*$, we may conclude that $\Delta_{BC} \omega^{n-1}$ lies in $\mathfrak{k}^* \wedge J\mathfrak{k}^* \wedge \Lambda^{2n-4} \mathfrak{n}_1^*$. Analogously, and using Remark \ref{Bal_avA}, we have $i\partial \overline{\partial} * \left(\rho^C \wedge \omega \right) \in \mathfrak{k}^* \wedge J\mathfrak{k}^* \wedge \Lambda^{2n-4} \mathfrak{n}_1^*$. Using the fact that $\omega^{n-2} \in \mathfrak{k}^* \wedge J\mathfrak{k}^* \wedge \Lambda^{2n-6} \mathfrak{n}_1^* \oplus  \Lambda^{2n-4} \mathfrak{n}_1^*$, our claim follows by contraction with $\omega^{n-2}$. With respect to a unitary basis $\{e_1,\ldots,e_{2n}\}$ adapted to the splitting $J\mathfrak{k} \oplus \mathfrak{n}_1 \oplus \mathfrak{k}$, we thus have
\[
q_{\mu}=-2p\,e^1 \wedge e^{2n} - 2 \omega(P \cdot, \cdot),
\]
for some $p=p(a,A) \in \R$ and $P=P(a,A) \in \mathfrak{gl}(\mathfrak{n}_1)$ which is symmetric and commutes with $J_1$, since both $\omega$ and $q_{\mu}$ are of type $(1,1)$ with respect to $J$, so that the claim follows by recalling \eqref{Qmu}.

Finally, the fact that $p$ and the entries of $P$ are homogeneous fourth order polynomials follows from the fact that the equation of the balanced flow is a fourth order \textsc{pde}.
\end{proof}

In dimension six, an explicit computation yields
\begin{equation} \label{p}
p(a,A)=\frac{1}{32}\left( \left(\operatorname{tr}(JA^2)\right)^2 - w(a,A)  \left\lVert A^+ \right\rVert^2 \right),
\end{equation}
and
\begin{equation} \label{P}
P(a,A)=\frac{w(a,A)}{32}\, [A,A^t]+\frac{a^3}{2}A^+ - \frac{a}{2} \,[A^-,A^tA],
\end{equation}
where $A^{\pm}=\frac{A \pm A^t}{2}$ and
\begin{equation} \label{w}
w(a,A)=-12a^2+ 4\left\lVert A \right\rVert^2- \left( \operatorname{tr}(JA) \right)^2.
\end{equation}
One can easily check that $A$ is traceless and commutes with $J_1=J_0\rvert_{\mathfrak{n}_1}$ if and only if its associated matrix with respect to the basis $\{e_2,\ldots,e_5\}$ is of the form
\begin{equation} \label{Amatrix}
A=\left( \begin{smallmatrix} A_{11} & A_{12} & A_{13} & A_{14} \\
A_{21} & -A_{11} & A_{23} & A_{24} \\
-A_{24} & -A_{23} & -A_{11} & A_{21} \\
-A_{14} & -A_{13} & A_{12} & A_{11} \end{smallmatrix} \right),
\end{equation}
for some $A_{ij} \in \R$.
For convenience, denote
\begin{equation} \label{b_i}
\begin{alignedat}{3}
b_1&=A_{11},\quad& b_2&=\tfrac{1}{2}(A_{12}+A_{21}),\quad& b_3&=\tfrac{1}{2}(A_{12}-A_{21}),\\
b_4&=\tfrac{1}{2}(A_{13}-A_{24}),\quad& b_5&=\tfrac{1}{2}(A_{13}+A_{24}),\quad& b_6&=\tfrac{1}{2}(A_{14}-A_{23}).
\end{alignedat}
\end{equation}
Then, one can check that the following formulas hold:
\begin{align*}
w(a,A)=&-12a^2 + 16(b_1^2+b_2^2+b_3^2+b_4^2+b_5^2+b_6^2), \\
p(a,A)=&\frac{3}{2} (b_1^2+b_2^2+b_4^2)a^2-2(b_1^2+b_2^2+b_4^2)^2-2(b_1b_5-b_2b_6)^2\\
      &-2(b_1b_3+b_4b_6)^2
-2(b_2b_3+b_4b_5)^2.
\end{align*}
As a consequence, we obtain $w(0,A) \geq 0$ and $p(0,A) \leq 0$. Moreover, $w(0,A)$ is strictly positive unless $A=0$ (corresponding to the abelian Lie bracket), while
\[
p(0,A)=0 \, \Leftrightarrow \, A_{11}=0,\,A_{12}+A_{21}=0,\,A_{13}-A_{24}=0\, \Leftrightarrow \, A \in \mathfrak{u}(\mathfrak{n}_1),
\]
which corresponds to the Hermitian almost abelian Lie algebra $(\mu(0,0,A),J,g)$ being K\"ahler.

\begin{definition}
A Lie group endowed with a left-invariant balanced structure $(G,J,g)$ is a \emph{semi-algebraic balanced soliton} if the endomorphism $Q=-\frac{1}{2}\omega^{-1}q(\omega) \in \mathfrak{gl}(\mathfrak{g})$ satisfies
\begin{equation} \label{semialg_soliton}
Q=\alpha\,\text{Id}_{\mathfrak{g}}+\tfrac{1}{2}(D+D^t),\quad \alpha \in \R,\,D \in \text{Der}(\mathfrak{g}) \cap \mathfrak{gl}(\mathfrak{g},J)
\end{equation}
and it is called \emph{algebraic balanced soliton} if $D$ satisfies \ref{semialg_soliton} and is symmetric with respect to $g$, so that one has $Q=\alpha\,\text{Id}_{\mathfrak{g}}+D$.
The soliton is called \emph{shrinking}, \emph{steady} or \emph{expanding} if the \emph{cosmological constant} $\alpha$ satisfies $\alpha > 0$, $\alpha=0$, $\alpha < 0$, respectively.
\end{definition}

The importance of the previous definition lies in the following result, whose proof is analogous to the one of \cite[Proposition 2.5]{AL}.

\begin{proposition}
Let $(G,J, g_0, \omega_0)$ be a semi-algebraic balanced soliton. Then the solution to the balanced flow on $(G,J)$ with initial data $\omega_0$ is of the form
\[
\omega(t)=c(t)\varphi_t^*\omega_0,
\]
for $c(t) \subset \R$ and $\varphi_t$ a one-parameter family of Lie automorphisms of $G$ and biholomorphisms of $(G,J)$.
\end{proposition}

In our case, we may easily obtain the following characterization for algebraic balanced solitons.

\begin{proposition} \label{prop_soliton}
Let $(J,\omega)$ be a left-invariant balanced structure on a non-nilpotent almost abelian Lie group $G$, with associated algebraic data $(a,A)$, $\operatorname{tr} A=0$, with respect to some fixed adapted unitary basis $\{e_1,\ldots,e_{2n}\}$. Then $(G,J,\omega)$ is an algebraic balanced soliton if and only if $A$ commutes with the symmetric endomorphism $P=Q_{\mu}\rvert_{\mathfrak{n}_1}$ in \eqref{Qmu_almostabelian}. In particular,  $Q=-\frac{1}{2}\omega^{-1}q(\omega) \in \mathfrak{gl}(\mathfrak{g})$ is given by \eqref{semialg_soliton}, with 
\begin{equation} \label{derivation}
\alpha=p,\quad D=P-p\,\text{\normalfont Id}_{\mathfrak{n}_1},
\end{equation}
where $p=\left<Q_{\mu}e_{2n},e_{2n}\right>$.
\end{proposition}
\begin{proof}
First, recall that a derivation of a non-nilpotent almost abelian Lie algebra is contained in the nilradical $\mathfrak{n}$. Thus, if we wish to decompose \eqref{Qmu_almostabelian} as \eqref{semialg_soliton}, assuming $D$ to be symmetric, one has to set the cosmological constant $\alpha$ and the derivation $D$ to be given by \eqref{derivation}.

Now, an endomorphism $D$ of $\mathfrak{n}_1$ is a derivation of $\mathfrak{g}$ if and only if it commutes with the $\operatorname{ad}_{e_{2n}}$-action on $\mathfrak{n}_1$: to see this, note that
\[
D[e_{2n},X]=[De_{2n},X]+[e_{2n},DX],\quad X \in \mathfrak{g},
\]
holds for $X=e_1,e_{2n}$ since both sides vanish (we are using that the balanced condition forces $v=0$), while, for $X \in \mathfrak{n}_1$, it reads $DA(X)=AD(X)$. The claim follows by substituting $D$ with $P-p\,\text{\normalfont Id}_{\mathfrak{n}_1}$, and using that $\text{Id}_{\mathfrak{n}_1}$ lies in the center of $\mathfrak{gl}_{2n}$.
\end{proof}

An important consequence of Proposition \ref{propQmu} is that, for any almost abelian Lie bracket $\mu(a,0,A)$, $\operatorname{tr} A=0$, the endomorphism $Q_{\mu(a,0,A)}$ preserves the decomposition $J\mathfrak{k} \oplus \mathfrak{n}_1 \oplus \mathfrak{k}$. In turn, this implies that the bracket flow \eqref{brflow} associated with the balanced flow preserves the set of brackets of the form $\mu(a,0,A)$ with respect to this decomposition, without the need to resort to a change of gauge. It is then natural to study the bracket flow in terms of \textsc{ode}s in the variables $a$ and $A$.

\begin{proposition}
Let $(\mu(a_0,0,A_0),J_0,\left< \cdot, \cdot \right>)$, $\operatorname{tr}A_0=0$, be a balanced almost abelian Lie algebra. Then the bracket flow \eqref{brflow} starting from the bracket $\mu(a_0,0,A_0)$ is equivalent to the {\normalfont \textsc{ode}} system
\begin{equation} \label{bracketODE}
\begin{cases}
\tfrac{d}{dt}a=p\,a, & a(0)=a_0, \\
\tfrac{d}{dt}A=[A,P]+p\,A, & A(0)=A_0,
\end{cases}
\end{equation}
where $p=p(a,A)$ and $P=P(a,A)$ are as in Proposition \ref{propQmu}.
\end{proposition}
\begin{proof}
Following \cite[Proposition 4.11]{AL} and looking at \eqref{brflow}, one computes
\[
\tfrac{d}{dt}\left(\text{ad}_{\mu}e_{2n}\right)=\text{ad}_{\dot{\mu}}e_{2n}=-\text{ad}_{\pi(Q_\mu)\mu}e_{2n}=-[Q_\mu,\text{ad}_{\mu}e_{2n}]+\text{ad}_{\mu}(Q_\mu e_{2n}),
\]
which yields
\[
\frac{d}{dt} \begin{pmatrix} a & 0 \\ 0 & A \end{pmatrix} = -\left[ \begin{pmatrix} p & 0 \\ 0 & P \end{pmatrix}, \begin{pmatrix} a & 0 \\ 0 & A \end{pmatrix} \right] + p\begin{pmatrix} a & 0 \\ 0 & A \end{pmatrix}=\begin{pmatrix} p\,a & 0 \\ 0 & [A,P]+p\,A \end{pmatrix}. \qedhere
\]
\end{proof}

\begin{theorem}
The maximal definition interval of the solution of the balanced flow  starting from a left-invariant balanced structure $(J,\omega_0)$ on a six-dimensional unimodular almost abelian Lie group $G$ with associated bracket $\mu(0,0,A_0)$, $\operatorname{tr} A_0=0$, is of the form $(T_{\text{\normalfont min}},\infty)$, for some $T_{\text{\normalfont min}} \in \R_{<0} \cup \{ -\infty \}$. Moreover, $\omega_0$ is a fixed solution of the balanced flow if and only if it is K\"ahler.
\end{theorem}
\begin{proof}
By the previous proposition, the balanced flow for left-invariant balanced metrics on unimodular almost abelian Lie groups is equivalent to the second equation in the \textsc{ode} system \eqref{bracketODE} for the associated algebraic datum $A$, since the other datum $a$ vanishes identically. We remark that, in a suitable basis for $\R^4$, the solution $A(t)$ of the bracket flow is of the form \eqref{Amatrix} for all $t$, with time-dependent entries. With respect to the scalar product 
\begin{equation} \label{scalmatrices}
\left<M,N\right>=\operatorname{tr}(M^t N)
\end{equation}
on matrices, we can compute
\begin{align*}
\tfrac{d}{dt} \lVert A \rVert^2&= 2 \left< A, \tfrac{d}{dt}A \right>=2p \lVert A \rVert^2 + 2 \left< A, [A,P] \right>.
\end{align*}
Now, using the explicit expression for $P$ in \eqref{P}, we obtain
\begin{equation} \label{<A,[A,P]>}
\left< A, [A,P] \right>=\frac{w(0, A)}{32} \left< A,[A, [A,A^t]] \right>
\end{equation}
Now, using \eqref{scalmatrices} and the fact that $\operatorname{tr}(MN)=\operatorname{tr}(NM)$ for every pair of matrices $M,N$, it is easy to see that $\left< A, [A,A^t] \right>=0$ and $\left< A,[A, [A,A^t]] \right>=-\lVert [A,A^t] \rVert^2$, while an explicit computation using a matrix $A$ of the form \eqref{Amatrix} shows that $\left< A, [A, [A^-,A^tA]] \right>=0$. Thus, \eqref{<A,[A,P]>} becomes
\[
\left< A, [A,P] \right> = -\frac{w(0, A)}{32}\lVert [A,A^t] \rVert^2,
\]
so that
\[
\tfrac{d}{dt} \lVert A \rVert^2= 2 p(0,A)\lVert A \rVert^2 - \frac{w(0, A)}{16} \lVert [A,A^t] \rVert^2,
\]
which is non-positive, using that $p(0,A) \leq 0$ and $w(0,A) \geq 0$. Thus, we can conclude that the solution to \eqref{bracketODE} remains inside a compact subset and is defined for all positive times. To prove the second part of the claim, it suffices to note that the previous expression is strictly negative unless $A \in \mathfrak{u}(\mathfrak{n}_1)$, which corresponds to a K\"ahler metric.
\end{proof}

In the following proposition, we say that a one-parameter family of Hermitian Lie groups $(G_t,J_t,g_t)$ (with identity element denoted by $e_t$) converges in the Cheeger-Gromov sense to a Hermitian Lie group $(G_{\infty},J_{\infty},g_{\infty})$ if there exist a subsequence $\{t_k\}_{k \in \mathbb{N}}$, a sequence of open subsets $U_k$ of $G_\infty$, which contain the identity element $e_{\infty}$ of $G_{\infty}$ and exhaust $G_\infty$, and a sequence of smooth embeddings $\varphi_k \colon U_k \to G_{t_k}$ such that $\varphi_k(e_\infty)=e_{t_k}$, $\varphi_k^*J_{t_k}=J_{\infty} \rvert_{U_k}$ and $\varphi_k^* g_{t_k}$ converges to $g_\infty$ in $C^{\infty}$ topology uniformly over compact subsets.

\begin{proposition}
Let $(J,\omega_0)$ be a left-invariant balanced structure on a six-dimensional unimodular almost abelian Lie group $G$ with associated bracket $\mu(0,0,A_0)$, $\operatorname{tr} A_0=0$. If the bracket flow \eqref{bracketODE} converges to a bracket $\mu(0,0,A_{\infty})$, then $A_{\infty} \in \mathfrak{u}(\mathfrak{n}_1)$. Moreover, the solution to the balanced flow \eqref{flowBV} converges in the Cheeger-Gromov sense to the flat K\"ahler unimodular almost abelian Lie group $(G,\mu(0,0,A_{\infty}),J_0,\left< \cdot, \cdot \right>)$. 
\end{proposition}
\begin{proof}
Let $(0,A(t))$ be the solution to \eqref{bracketODE} starting from $(0,A_0)$: one can compute
\begin{align*}
\tfrac{d}{dt} \lVert A^+ \rVert^2 = \frac{1}{2} \left< \tfrac{d}{dt}(A+A^t), A+A^t \right>=p\lVert A^+ \rVert^2 + \left< [A^-,P],A^+\right>,
\end{align*}
which an explicit computation using \eqref{p}--\eqref{b_i} shows to yield
\begin{align*}
\tfrac{d}{dt} \lVert A^+ \rVert^2=&-\frac{1}{4} \lVert A^+ \rVert^6-16\left((b_1b_3+b_4b_6)^2+(b_1b_5-b_2b_6)^2+(b_2b_3+b_4b_5)^2\right) \\&\cdot\left(3(b_1^2+b_2^2+b_4^2)+2(b_3^2+b_5^2+b_6^2)\right)-4a^2(b_1^2+b_2^2+b_4^2)^2.
\end{align*}
Thus, we have
\[
\tfrac{d}{dt} \lVert A^+ \rVert^2\leq - \frac{1}{4} \lVert A^+ \rVert^6.
\]
A comparison with the \textsc{ode} $\frac{d}{dt} y = -\frac{1}{4} y^3$, $y(0)=\lVert A^+_0 \rVert^2$, yields the a priori estimate
\[
\lVert A^+(t) \rVert^2 \leq \left( \lVert A^+_0 \rVert^{-4} + \frac{t}{2} \right)^{-\frac{1}{2}},
\]
so that $\lim_{t \to \infty} A^+(t) =0$. It follows that $A_\infty \in \mathfrak{u}(\mathfrak{n}_1)$, as claimed.

The second part of the claim follows by well-known results (see \cite[Corollary 6.20]{Lau1} and \cite[Section 5.1]{Lau}). The limit K\"ahler Lie group is Ricci-flat (hence flat, by \cite{AK}), as one can check using formula \eqref{rhoC}, for example.
\end{proof}

\begin{example}
We now  study   the balanced flow  on the  six-dimensional nilpotent almost abelian Lie algebra $(0,0,0,0,f^{12},f^{13})$
endowed  with the complex structure determined by a basis of $(1,0)$-forms $\{\alpha^1,\alpha^2,\alpha^3\}$ satisfying
\[
d\alpha^1=d\alpha^2=0,\quad d\alpha^3=i\alpha^{12}+i\alpha^{2\overline{1}}
\]
and  the balanced Hermitian metric with fundamental form
\[
\omega_0=i\alpha^{1\overline{1}}+i\alpha^{2\overline{2}}+i\alpha^{3\overline{3}}.
\]
Let us assume the solution to the balanced flow \eqref{flowBV2} is of the form
\[
\omega(t)=ig_1(t)\alpha^{1\overline{1}}+ig_2(t)\alpha^{2\overline{2}}+ig_3(t)\alpha^{3\overline{3}},
\]
for some smooth functions $g_i(t)$, $g_i(0)=1$, $i=1,2,3$. Then one can compute
\[
\Delta_{BC} \omega^2 = -2 \partial \overline\partial * \partial \overline\partial \omega = -8 \frac{g_3^2}{g_1g_2}\alpha^{1\overline{1}2\overline{2}}
\]
and $\iota_{\omega}\left( \tfrac{1}{2} \Delta_{BC} \omega^2 \right)$, the unique $2$-form whose wedge product with $\omega$ gives $\tfrac{1}{2}\Delta_{BC} \omega^2 $, is given by
\[
\iota_{\omega}\left( \tfrac{1}{2} \Delta_{BC} \omega^2 \right)=2i \frac{g_3^2}{g_1g_2^2} \alpha^{1\overline{1}} +2i \frac{g_3^2}{g_1^2g_2} \alpha^{2\overline{2}}-2i \frac{g_3^3}{g_1^2g_2^2} \alpha^{3\overline{3}},
\]
so that the balanced flow is equivalent to
\[
\begin{dcases}
	\tfrac{d}{dt} g_1= 2\frac{g_3^2}{g_1g_2^2},\\
	\tfrac{d}{dt}g_2= 2\frac{g_3^2}{g_1^2g_2},\\
	\tfrac{d}{dt}g_3=-2\frac{g_3^3}{g_1^2g_2^2},
\end{dcases}
\]
with initial conditions $g_1(0)=g_2(0)=g_3(0)=1$, whose solution is
\[
g_1(t)=g_2(t)=(1+12t)^{\frac{1}{6}},\quad g_3(t)=(1+12t)^{-\frac{1}{6}},
\]
defined on $I=(-\frac{1}{12},\infty)$.

In terms of bracket flow, we can see that this is in fact an expanding algebraic soliton. Let $(e_1,\ldots,e_6)$ be the unitary basis of $\mathfrak{g}$ for $(J,\omega_0)$, such that $\alpha^1=\frac{\sqrt{2}}{2}(e^1+ie^6)$, $\alpha^2=\frac{\sqrt{2}}{2}(e^2+ie^5)$, $\alpha^3=\frac{\sqrt{2}}{2}(e^3+ie^4)$, $Je_i=e_{7-i}$, $i=1,2,3$. The corresponding data $(a,0,A)$ and endomorphism $Q_{\mu(a,0,A)}$ are
\[
a=0,\quad A=\left( \begin{smallmatrix} 0&0&0&0 \\ \sqrt{2}&0&0&0 \\ 0&0&0&\sqrt{2} \\ 0&0&0&0 \end{smallmatrix} \right), \quad Q_{\mu(a,0,A)}=\left( \begin{smallmatrix}
-1&0&0&0&0&0 \\
 0&-1&0&0&0&0 \\
 0&0&1&0&0&0 \\
 0&0&0&1&0&0 \\
 0&0&0&0&-1&0 \\
 0&0&0&0&0&-1 
\end{smallmatrix} \right).
\]
We have $Q_{\mu(a,0,A)}=\alpha\,\text{Id}_{6}+D$, with
\[
\alpha=-3,\quad D=\left( \begin{smallmatrix} 
	2&0&0&0&0&0 \\
	0&2&0&0&0&0 \\
	0&0&4&0&0&0 \\
	0&0&0&4&0&0 \\
	0&0&0&0&2&0 \\
	0&0&0&0&0&2   \end{smallmatrix} \right) \in \text{Der}(\mu(a,0,A)) \cap \text{Sym}_6.
\]
In contrast with Proposition \ref{prop_soliton}, which holds for non-nilpotent almost abelian Lie algebras, we have $[A,P] \neq 0$, with $P=Q_{\mu(a,0,A)}\rvert_{\mathfrak{n}_1}$, $\mathfrak{n}_1=\text{span}\left<e_2,e_3,e_4,e_5\right>$. \end{example}

\section{Anomaly flow} \label{sec_anom}

Another important geometric flow in (conformally) balanced geometry is the so-called \emph{anomaly flow}, introduced in \cite{PPZ, PPZ2}. 

Consider a complex manifold $X$ of complex dimension $3$, endowed with a nowhere vanishing closed $(3,0)$-form $\Psi$, and a holomorphic vector bundle $E$ over $X$. The anomaly flow can be written as the coupled flow for $(\omega(t),H(t))$, Hermitian metrics respectively on $X$ and on the fibers of $E$,
\[
\begin{dcases}
\tfrac{\partial}{\partial t} \left(\lVert \Psi \rVert_{\omega(t)} \omega(t)^2	\right) = i \partial \overline \partial \omega(t) - \frac{\alpha^\prime}{4} \left( \operatorname{tr}( R^{\tau}_t \wedge R^{\tau}_t ) - \operatorname{tr}(F_t^{\kappa} \wedge F_t^{\kappa})\right), & \omega(0)=\omega_0,\\
H(t)^{-1} \tfrac{\partial}{\partial t} H(t) = \frac{\omega(t)^2 \wedge F_t^{\kappa}}{\omega(t)^3}, & H(0)=H_0,
\end{dcases}
\]
where $R^\tau_t$ and $F^\kappa_t$ are the curvatures of the Gauduchon connections $\nabla^\tau$ and $\nabla^{\kappa}$ of $(X,\omega(t))$ and $(E,H(t))$ respectively, $\alpha^\prime \in \R$ is the so-called \emph{slope parameter} and the norm of $\Psi$ is defined by
\[
i \Psi \wedge \overline{\Psi}= \lVert \Psi \rVert^2_{\omega} \,\omega^3.
\]

The anomaly flow arises in relation to the Hull-Strominger system \cite{Hull, Str},
\[
	\begin{aligned}
		&F^{\kappa} \wedge \omega^2=0,\\
		&(F^{\kappa})^{2,0}=(F^{\kappa})^{0,2}=0,\\
		&i \partial \overline \partial \omega= \frac{\alpha^\prime}{4}\left( \operatorname{tr}( R^{\tau} \wedge R^{\tau} ) - \operatorname{tr}(F^{\kappa} \wedge F^{\kappa})\right),\\
		&d(\lVert \Psi \rVert_{\omega} \omega^2 )=0,
	\end{aligned}
\]
which follows from imposing the supersymmetry conditions for the internal space of heterotic string theory in ten dimensions.

Since its introduction, the anomaly flow has been studied on some classes of complex manifolds, see for example \cite{PPZ3, FHP,PU}.
In \cite{PPZ2}, a simplified version of the anomaly flow was introduced, given by the flow of Hermitian metrics
\begin{equation} \label{anomaly2}
\tfrac{\partial}{\partial t} (\lVert \Psi \rVert_{\omega(t)} \omega(t)^2 ) = i \partial \overline \partial \omega(t) - \frac{\alpha^\prime}{4} \operatorname{tr}(R^{\tau}_t \wedge R^{\tau}_t ), \quad \omega(0)=\omega_0.
\end{equation}

The flow \eqref{anomaly2} is equivalent to the flow for $\tilde{\omega}(t)=\lVert \Psi \rVert^{\frac{1}{2}}_{\omega(t)} \omega(t)$,
\begin{equation} \label{anomaly_rescal}
\tfrac{\partial}{\partial t} \tilde{\omega}(t)^2= \lVert \Psi \rVert^{-2}_{\tilde{\omega}(t)} i \partial \overline \partial \tilde{\omega}(t) - \frac{\alpha^\prime}{4}\operatorname{tr}(R^{\tau}_t \wedge R^{\tau}_t ),\quad \tilde{\omega}(0)=\lVert \Psi \rVert^{\frac{1}{2}}_{\omega_0} \omega_0,
\end{equation}
or equivalently
\[
\tfrac{\partial}{\partial t} \tilde{\omega}(t)= \frac{1}{2} \lVert \Psi \rVert^{-2}_{\tilde{\omega}(t)} \iota_{\tilde{\omega}(t)} (i \partial \overline \partial \tilde{\omega}(t)) - \frac{\alpha^\prime}{8} \iota_{\tilde{\omega}(t)}\left(\operatorname{tr}(R^{\tau}_t \wedge R^{\tau}_t ) \right), \quad \tilde{\omega}(0)=\lVert \Psi \rVert^{\frac{1}{2}}_{\omega_0} \omega_0.
\]

Let $(\mathfrak{g}(a,v,A),J,g)$ be a Hermitian almost abelian Lie algebra, with fixed unitary basis $\{e_1,\ldots,e_{2n}\}$ adapted to the splitting $\mathfrak{g}=J\mathfrak{k} \oplus \mathfrak{n}_1 \oplus \mathfrak{k}$.

We proceed by discussing the curvature term in \eqref{anomaly2}. First, we examine the case of the Chern connection for generic Hermitian almost abelian Lie algebras. Then, we go back to the generic Gauduchon connection for balanced and LCK almost abelian Lie algebras.

Recalling the expression \eqref{nablatau} for the connections in the Gauduchon line, by a direct computation we first obtain the next result (cf. \cite[Lemma 4.1]{FTa} for the case $\tau=-1$).

\begin{proposition} \label{prop_nablatau}
The canonical Hermitian connection $\nabla^\tau$ of a Hermitian almost abelian Lie algebra $(\mathfrak{g}(a,v,A),J,g)$ is given by
\begin{align*}
	\nabla^{\tau}_{e_1} e_1&= a e_{2n} + \frac{1+\tau}{4} Jv, \\
	\nabla^{\tau}_{e_1} Y&= \frac{1+\tau}{4} \left( - g(Jv,Y)e_1 + g(v,Y)e_{2n} \right) + \tau A^+ JY, \\
	\nabla^{\tau}_{e_1} e_{2n} &= -ae_1 - \frac{1+\tau}{4} v,\\
	\nabla^{\tau}_{X} e_1&= \frac{1-\tau}{2} \left( A^+JX + g(v,X) e_{2n} \right), \\
	\nabla^{\tau}_X Y &=- \frac{1-\tau}{2} \left( g(A^+JX,Y) e_1 - g(A^+X,Y) e_{2n} \right),\\
	\nabla^{\tau}_{X} e_{2n}&= - \frac{1-\tau}{2} \left( g(v,X) e_1 + A^+X\right),\\
	\nabla^{\tau}_{e_{2n}} e_1&= \frac{1+\tau}{4} v,\\
	\nabla^{\tau}_{e_{2n}} Y&= - \frac{1+\tau}{4} g(v,Y) e_1 + A^-Y - \frac{1+\tau}{4} g(Jv,X) e_{2n},\\
	\nabla^{\tau}_{e_{2n}} e_{2n} &= \frac{1+\tau}{4} Jv,
\end{align*}
where $X,Y \in \mathfrak{n}_1$ and $\{e_1,\ldots,e_{2n}\}$ denotes the fixed adapted unitary basis. In particular, for the Chern connection $(\tau=1)$ one has $\nabla^C_X=0$,  for all $X \in \mathfrak{n}_1$.
\end{proposition}

The fact that $\nabla^C_X=0$, for all $X \in \mathfrak{n}_1$,  also follows directly from \eqref{nablatau}, by exploiting the fact that $\text{ad}_Y \rvert_{\mathfrak{n}_1}$ is $J$-invariant for all $Y \in \mathfrak{g}$: indeed, if $X \in \mathfrak{n}_1$, $Y,Z \in \mathfrak{g}$, recalling the Koszul formula \eqref{koszul} we have
\begin{align*}
C(X,Y,Z)&=-d\omega(JX,Y,Z)=g(J[JX,Y],Z)-g(J[JX,Z],Y)+g(J[Y,Z],JX) \\        &=-g([X,Y],Z)+g([X,Z],Y)+g([Y,Z],X)=-2g(\nabla_X^g Y,Z),
\end{align*}
so that
\[
g(\nabla^C_X \cdot, \cdot) = g(\nabla^g_X \cdot, \cdot) + \frac{1}{2} C(X,\cdot,\cdot)=0, \quad X \in \mathfrak{n}_1.
\]

The connection $1$-forms $(\sigma^{\tau})^i_j \in \mathfrak{g}^*$ and the curvature $2$-forms $(\Omega^{\tau})^i_j$, $i,j=1,\ldots,2n$, are defined respectively by
\begin{align*}
\nabla^\tau_X &= (\sigma^{\tau})^i_j(X) \; e^j \otimes e_i, \\
R^{\tau}(X,Y) &= (\Omega^{\tau})^i_j(X,Y) \; e^j \otimes e_i,
\end{align*}
for $X,Y \in \mathfrak{g}$, with respect to a fixed orthonormal frame $\{e_1,\ldots,e_{2n}\}$,. The latter are related to the former by the Cartan structure relations
\begin{equation} \label{Omegatau}
(\Omega^{\tau})^i_j= d(\sigma^{\tau})^i_j + (\sigma^{\tau})^i_k \wedge (\sigma^{\tau})^k_j.
\end{equation}
We can define
\begin{equation} \label{trcurv}
\operatorname{tr}(R^{\tau} \wedge R^{\tau})=\sum_{i<j} \Omega^i_j \wedge \Omega^i_j.
\end{equation}

An immediate consequence of Proposition \ref{prop_nablatau} is the following fact, which provides a great simplification of the anomaly flow equations, when $\tau=1$.

\begin{corollary} 
The curvature $R^C$  of the Chern connection of a Hermitian almost abelian Lie algebra $(\mathfrak{g}(a,v,A),J,g)$ satisfies $R^C(X,Y)=0$ for all $X \in \mathfrak{n}_1$, $Y \in \mathfrak{g}$, that is,
\[
R^C \in \mathfrak{k}^* \wedge J\mathfrak{k}^* \otimes \text{\normalfont End}(\mathfrak{g}),
\]
where $\mathfrak{k}=\mathfrak{n}^{\perp_g}$. As a consequence,
\[
R^C \wedge R^C = 0.
\]
\end{corollary}
\begin{proof}
Following the splitting $\mathfrak{g}=J\mathfrak{k} \oplus \mathfrak{n}_1 \oplus \mathfrak{k}$, inducing
\[
\Lambda^2 \mathfrak{g}^*=\mathfrak{k}^* \wedge J\mathfrak{k}^* \oplus \Lambda^2 \mathfrak{n}_1^* \oplus \mathfrak{n}_1^* \wedge \mathfrak{k}^* \oplus \mathfrak{n}_1^* \wedge J\mathfrak{k}^*,
\]
the first part of the claim follows by
\[
R^C(X,Y)=\nabla^C_X \nabla^C_Y-\nabla^C_Y \nabla^C_X - \nabla^C_{[X,Y]} = 0, \quad X \in \mathfrak{n}_1,\,Y \in \mathfrak{g}, 
\]
where we have used $\nabla^C_X=0$ and the fact that $[X,Y] \in \mathfrak{n}_1$. Then,
\[
R^C \wedge R^C \in \Lambda^4(\mathfrak{k} \oplus J\mathfrak{k})^* \otimes \text{End}(\mathfrak{g})=0. \qedhere
\]
\end{proof}

When $\tau \neq 1$, there is no guarantee that the flow \eqref{anomaly2} is well-defined, since the right-hand side might fail to be a $(2,2)$-form. In \cite{FY}, it was proven that, for unimodular complex Lie groups (that is, unimodular real Lie groups with bi-invariant complex structures), this is never a problem and the aforementioned trace term is always of type $(2,2)$. Recently, in \cite{PU}, the analogous result was shown to hold for $2$-step nilpotent Lie groups with first Betti number  $b_1 \geq 4$ endowed with left-invariant Hermitian structures.

In the almost abelian case, instead, $\operatorname{tr}(R^{\tau} \wedge R^{\tau} )$ is not a $(2,2)$-form, in general, for $\tau \neq 1$. It is, though, under more restrictive hypotheses on the Hermitian structure one is considering. Recall that, combining Propositions \ref{closed(n,0)} and \ref{Bal_avA}, a Hermitian almost abelian Lie algebra $(\mathfrak{g}(a,v,A),J,g)$ is balanced and admits a closed $(n,0)$-form if and only if
\[
v=0,\quad a=\operatorname{tr} A = \operatorname{tr} JA = 0.
\]
In particular, the Lie algebra $\mathfrak{g}$ is decomposable and unimodular.

\begin{lemma}
Let $(\mathfrak{g}(0,0,A),J,g)$, $\operatorname{tr} A=0$, be a six-dimensional balanced almost abelian Lie algebra admitting a closed $(3,0)$-form. Then, for any $\tau \in \R$, one has
\[
\operatorname{tr}(R^\tau \wedge R^\tau) \in \Lambda^{2,2} \mathfrak{g}^*.
\]
More precisely,
\[
\operatorname{tr}(R^\tau \wedge R^\tau)=\omega \wedge (\omega(L \cdot, \cdot)),
\]
where $L=L(A,\tau)$ is a symmetric endomorphism of $\mathfrak{g}$ commuting with $J$, given by
\[
L= \frac{1}{8} \tau (1-\tau)^2 \lVert A^+ \rVert^2 \begin{pmatrix} 
\lVert A^+ \rVert^2 & 0 & 0 \\
    0 & -[A,A^t] & 0 \\
    0 & 0 & \lVert A^+ \rVert^2
 \end{pmatrix},
\]
with respect to the fixed adapted unitary basis. In particular,
\[
\iota_{\omega} \operatorname{tr}(R^\tau \wedge R^\tau)= \omega(L \cdot, \cdot) \in \mathfrak{k}^* \wedge J\mathfrak{k}^* \oplus \Lambda^2 \mathfrak{n}_1^*.
\]
\end{lemma}
\begin{proof}

Using Proposition \ref{prop_nablatau} and formula \eqref{Omegatau}, one can compute the curvature $2$-forms
\begin{alignat*}{2}
	\Omega^1_j =& \tfrac{1}{8} \tau (1-\tau) \lVert A^+ \rVert^2 \, e^{1j} + \tfrac{1}{4} (1-\tau) \left( \left( 2(A^+)^2 - [A,A^t] \right)Je_j \right)^\flat \wedge e^6, \qquad&& j=2,\ldots,5,\\
	\Omega^1_6 =& \tfrac{1}{8} (1-\tau)^2 \lVert A^+ \rVert^2 \, \omega\rvert_{\mathfrak{n}_1}, \\
	\Omega^i_j =& \tfrac{1}{2} \tau \omega \left( [A,A^t]e_i, e_j \right) e^{16} - \tfrac{1}{4} (1-\tau)^2 \left( (A^+e_i)^\flat \wedge (A^+e_j)^\flat \right)^{1,1}, \qquad&& i,j=2,\ldots,5,  \\
	\Omega^i_6 =&\tfrac{1}{8} \tau(1-\tau) \lVert A^+ \rVert^2 \,e^1 \wedge Je^i - \tfrac{1}{4}(1-\tau) \left( \left( 2(A^+)^2 - [A,A^t] \right)e_i \right)^\flat \wedge e^6,\qquad&& i=2,\ldots,5.
\end{alignat*}
where $\alpha^{1,1}\coloneqq\frac{1}{2}(\alpha+J\alpha)$, $\alpha \in \Lambda^2 \mathfrak{g}^*$. One can use these to find $\operatorname{tr}(R^\tau \wedge R^\tau)$ via \eqref{trcurv} and directly verify the claim.
\end{proof}

As a consequence, we obtain the following

\begin{theorem} \label{anomaly_bal}
The anomaly flow \eqref{anomaly2} on almost abelian Lie groups preserves the balanced condition for all values of the parameters $\tau, \alpha \in \R$, in the left-invariant case.
\end{theorem}
\begin{proof}
From the previous proposition, together with the fact that $\iota_\omega (i \partial \overline \partial \omega)$ belongs to $\mathfrak{k}^* \wedge J\mathfrak{k}^* \oplus \Lambda^2 \mathfrak{n}_1^*$ (see the proof of Proposition \ref{propQmu}), we readily obtain that the solution $\tilde{\omega}(t)$ to the flow \eqref{anomaly_rescal} preserves the orthogonal splitting $\mathfrak{g}=J\mathfrak{k} \oplus \mathfrak{n}_1 \oplus \mathfrak{k}$ induced by the initial data. Therefore, a suitable adapted unitary basis for $(J,\tilde{\omega}(t))$ will induce algebraic data $(a(t),v(t),A(t))$ satisfying $a(t)=0$, $v(t)=0$, $\operatorname{tr} A(t)=\operatorname{tr} f(t)A(0)=0$ (for some smooth function $f(t)$), implying that $(J,\tilde{\omega}(t))$ is balanced for all $t$. The claim follows from the fact that the solutions to \eqref{anomaly2} and \eqref{anomaly_rescal} differ only by time-dependent rescalings.
\end{proof}

\begin{remark}
The results of Theorem \ref{anomaly_bal}, as well as Proposition \ref{closed(n,0)} and Theorem \ref{Bal_avA}, were obtained independently by M. Pujia \cite{Puj}, using different methods.
\end{remark}	

In the LCK case, instead, we have the following much stronger fact.

\begin{theorem}
Left-invariant LCK metrics on almost abelian Lie groups endowed with a left-invariant complex structure are fixed points of the anomaly flow \eqref{anomaly2}, namely they provide constant solutions of \eqref{anomaly2}, for all values of the parameters $\tau, \alpha \in \R$.
\end{theorem}
\begin{proof}
Let $(\mathfrak{g}(a,v,A),J,g)$ be a LCK almost abelian Lie algebra of complex dimension $3$ admitting closed $(3,0)$-forms and let $\{e_1,\ldots,e_6\}$ denote the fixed adapted unitary basis. Then, as shown in Proposition \ref{LCK_SKT}, the LCK structure is SKT and the algebraic data satisfies
\[
v=0,\quad A=-\tfrac{a}{2} \text{Id}_{\mathfrak{n}_1} + U,\; U \in \mathfrak{u}(\mathfrak{n}_1,J_1,g).
\]
We need to prove that the right hand side of \eqref{anomaly2} vanishes. The first summand is clearly zero by definition of SKT metric, so that we are left with the curvature term. Using Proposition \ref{prop_nablatau}, the only non-zero components of $R^\tau$, up to symmetries, can be summarized as follows:
\begin{align*}
g(R^\tau (e_1,X)e_1,Y)&=-\tfrac{1}{8} a^2 (\tau+2)(1-\tau)g(X,Y),\\
g(R^\tau (e_1,e_6)e_1,e_6)&=a^2,\\
g(R^\tau (e_1,e_6)X,Y)&=-\tfrac{1}{2} a^2 \tau \omega(JX,Y),\\
g(R^\tau (X,e_6)Y,e_6)&=\tfrac{1}{8} a^2 (1-\tau)g(X,Y),\\
g(R^\tau (X,Y)X,Y)&=\tfrac{1}{16} a^2 (1-\tau)^2\left(1+\omega(X,Y)^2-g(X,Y)^2\right),\\
g(R^\tau (X,Y)e_1,e_6)&=-\tfrac{1}{8} a^2 (1-\tau)^2\omega(X,Y),
\end{align*}
with $X,Y \in \mathfrak{n}_1$, $\lVert X \rVert=\lVert Y \rVert=1$, yielding curvature $2$-forms
\begin{alignat*}{2}
	\Omega^1_j &= \tfrac{1}{8} a^2 (\tau+2)(1-\tau) e^{1j} + \tfrac{1}{8} a^2 (1-\tau) \, Je^j \wedge e^6,\qquad&& j=2,\ldots,5, \\
	\Omega^1_6 &= -a^2 e^{16} + \tfrac{1}{8} a^2 (1-\tau)^2 \omega \rvert_{\mathfrak{n}_1}, && \\
	\Omega^i_j &= \tfrac{1}{2}a^2\tau \omega(e_i,e_j) e^{16} -\tfrac{1}{8} a^2 (1-\tau)^2 (e^{ij})^{1,1},\qquad&& i,j=2,\ldots,5, \\
	\Omega^i_6 &= - \tfrac{1}{8} a^2 (1-\tau) e^{i6} - \tfrac{1}{8} a^2 (1-\tau)(\tau+2) \, e^1 \wedge Je^i,\qquad&& i=2,\ldots,5,
\end{alignat*}
which one can use to explicitly compute $\operatorname{tr}(R^\tau \wedge R^\tau)=0$, concluding the proof.
\end{proof}

\smallskip

\end{document}